\documentclass[12pt,a4paper,oneside]{article}

\usepackage[left=2.5cm,right=2.5cm,top=2.5 cm,bottom=2.5cm]{geometry}

\usepackage[utf8]{inputenc}
\usepackage[T1]{fontenc}
\usepackage[english]{babel}
\usepackage{tikz}
\usepackage{pgfplots}
\pgfplotsset{compat=newest}
\usepgfplotslibrary{fillbetween}
\usepackage{hyperref}
\hypersetup{
    colorlinks = true,
    linkcolor = blue,
    anchorcolor = blue,
    citecolor = blue,
    filecolor = blue,
    urlcolor = blue
    }
\usetikzlibrary{decorations.markings,arrows}
\usepackage{breqn}
\usepackage{multirow}
\usepackage{amsmath}
\usepackage{longtable}
\usepackage{amsfonts}
\usepackage{amssymb}
\usepackage{graphicx}
\usepackage{amsthm}
\allowdisplaybreaks 
\usepackage{systeme}
\usepackage{float}
\usepackage{enumerate}
\usepackage{mathtools}
\usepackage{comment}
\usepackage{enumerate}
\usepackage{makecell} 
\usepackage{tcolorbox}
\usepackage{collectbox}

\usepackage{etoolbox}
\let\bbordermatrix\bordermatrix
\patchcmd{\bbordermatrix}{\left(}{\left[}{}{}
\patchcmd{\bbordermatrix}{\right)}{\right]}{}{}
\newtheorem{tm}{Theorem} 
\newtheorem{lm}{Lemma} 
\newtheorem{cor}{Corollary}

\numberwithin{equation}{section}

\newcolumntype{C}[1]{>{\centering\arraybackslash$}m{#1}<{$}}

\newcolumntype{R}[1]{>{\raggedleft\arraybackslash$}m{#1}<{$}}

\newcolumntype{L}[1]{>{\raggedright\arraybackslash$}m{#1}<{$}}

\DeclarePairedDelimiter\floor{\lfloor}{\rfloor}

\DeclarePairedDelimiterX{\cif}[1]{(}{)}{\delimsize(#1\delimsize)}

\begin{document}
\title{Tilings of a Honeycomb Strip and Higher Order Fibonacci Numbers}

\author{Tomislav Došlić$^{1,2}$, Luka Podrug$^1,$\footnote{Corresponding author.}\\ \\
{\small     $^1$ University of Zagreb, Faculty of Civil Engineering}\\
{\small     Ka\v{c}i\'ceva 26, 10000 Zagreb, Croatia} \\
{\small $^2$ Faculty of Information Studies, Novo mesto, Slovenia} \\ \\
{\small   \text{\{\tt tomislav.doslic,luka.podrug\}@grad.unizg.hr}}
}
\date{}
\maketitle

\begin{abstract}
\noindent
In this paper we explore two types of tilings of a honeycomb strip and
derive some closed form formulas for the number of tilings. Furthermore,
we obtain some new identities involving tribonacci numbers, Padovan numbers
and Narayana's cow sequence and provide combinatorial proofs for several
known identities about those numbers.

\vspace{4mm}

\noindent{\bf Keywords}: tilings; hexagonal lattice; tribonacci numbers;
                         tetranacci numbers; Padovan numbers; Fibonacci
                         numbers; Narayana's cow sequence
\end{abstract}

\section{Introduction}

Tilings or tesselations have been attracting human attention since the time
immemorial. They appear as natural solutions of many practical problems and
their aesthetic appeal motivates the interest that goes way
beyond the limits of their practical relevance. In mathematics, 
tiling-related problems appear in almost all areas, ranging from purely
recreational settings of plane geometry all the way to the deep questions of 
eigenvalue count asymptotics for boundary-value problems in higher-dimensional
spaces \cite{polya54,polya61}. Many of those problems, formulated in simple
and intuitive terms and seemingly innocuous, quickly turn out to be quite
intractable in their generality. That motivates interest in their restricted
versions that might be more accessible. In this paper we look at several such
restricted problems when the area being tiled has a given structure and the
allowed tiles belong to a small set of given shapes. In particular, we
consider problems of tiling a narrow strip of the hexagonal lattice in the 
plane by several types of tiles made of regular hexagons. Similar problems
for strips in square and triangular lattices have been considered in several
recent papers \cite{BBKSW,Ziqian,DT,KS}.

The substrate (i.e., the area to be tiled) is a honeycomb strip $H_n$ composed of $n$ regular hexagons arranged in two rows in which the hexagons are numbered starting from the bottom left corner, as shown in Figure \ref{grid1}
for $n=12$. The number of hexagons in the strip will be called its length.
The choice of the substrate might seem arbitrary, but it provides a neat
visual model for a linear array of locally interacting units with additional
longer range connections: The inner dual of a strip of length $n$ is, in
fact, $P_n^2$, the path on $n$ vertices with edges between all vertices at
distance 2 in $P_n$. Another way to look at it is as the ladder graph
with descending diagonals, another familiar structure. Clearly, tilings
with monomers and dimers in the strip correspond to matchings in its
inner dual, thus enabling us to directly transfer known results about
matchings into our context. We refer the reader to the classical monograph
by Lov\'asz and Plummer \cite{LP} for all necessary details on matchings.

\begin{figure}[h!]
\centering

\begin{tikzpicture}[scale=0.55]
\coordinate (A) at (0,1);
\coordinate (B) at ({-sqrt(3)/2},0.5);
\coordinate (C) at ({-sqrt(3)/2},-0.5);   
\coordinate (D) at (0,-1);
\coordinate (E) at ({sqrt(3)/2},-0.5);
\coordinate (F) at ({sqrt(3)/2},0.5);

\coordinate (A1) at ({sqrt(3)},1);
\coordinate (B1) at ({sqrt(3)/2},0.5);
\coordinate (C1) at ({sqrt(3)/2},-0.5);   
\coordinate (D1) at ({sqrt(3)},-1);
\coordinate (E1) at ({3*sqrt(3)/2},-0.5);
\coordinate (F1) at ({3*sqrt(3)/2},0.5); 

\coordinate (A2) at ({2*sqrt(3)},1);
\coordinate (B2) at ({3*sqrt(3)/2},0.5);
\coordinate (C2) at ({3*sqrt(3)/2},-0.5);   
\coordinate (D2) at ({2*sqrt(3)},-1);
\coordinate (E2) at ({5*sqrt(3)/2},-0.5);
\coordinate (F2) at ({5*sqrt(3)/2},0.5); 

\coordinate (A3) at ({3*sqrt(3)},1);
\coordinate (B3) at ({5*sqrt(3)/2},0.5);
\coordinate (C3) at ({5*sqrt(3)/2},-0.5);   
\coordinate (D3) at ({3*sqrt(3)},-1);
\coordinate (E3) at ({7*sqrt(3)/2},-0.5);
\coordinate (F3) at ({7*sqrt(3)/2},0.5);

\coordinate (A4) at ({4*sqrt(3)},1);
\coordinate (B4) at ({7*sqrt(3)/2},0.5);
\coordinate (C4) at ({7*sqrt(3)/2},-0.5);   
\coordinate (D4) at ({4*sqrt(3)},-1);
\coordinate (E4) at ({9*sqrt(3)/2},-0.5);4
\coordinate (F4) at ({9*sqrt(3)/2},0.5);

\coordinate (A5) at ({5*sqrt(3)},1);
\coordinate (B5) at ({9*sqrt(3)/2},0.5);
\coordinate (C5) at ({9*sqrt(3)/2},-0.5);   
\coordinate (D5) at ({5*sqrt(3)},-1);
\coordinate (E5) at ({11*sqrt(3)/2},-0.5);4
\coordinate (F5) at ({11*sqrt(3)/2},0.5);

\coordinate (A6) at ({6*sqrt(3)},1);
\coordinate (B6) at ({11*sqrt(3)/2},0.5);
\coordinate (C6) at ({11*sqrt(3)/2},-0.5);   
\coordinate (D6) at ({6*sqrt(3)},-1);
\coordinate (E6) at ({13*sqrt(3)/2},-0.5);4
\coordinate (F6) at ({13*sqrt(3)/2},0.5);

\coordinate (G) at ({sqrt(3)/2},-0.5);
\coordinate (H) at ({0},-1);
\coordinate (I) at ({0},-2);   
\coordinate (J) at ({sqrt(3)/2},-2.5);
\coordinate (K) at ({sqrt(3)},-2);
\coordinate (L) at ({sqrt(3)},-1);

\coordinate (G1) at ({3*sqrt(3)/2},-0.5);
\coordinate (H1) at ({{sqrt(3)}},-1);
\coordinate (I1) at ({{sqrt(3)}},-2);   
\coordinate (J1) at ({3*sqrt(3)/2},-2.5);
\coordinate (K1) at ({2*sqrt(3)},-2);
\coordinate (L1) at ({2*sqrt(3)},-1);

\coordinate (G2) at ({5*sqrt(3)/2},-0.5);
\coordinate (H2) at ({{2*sqrt(3)}},-1);
\coordinate (I2) at ({{2*sqrt(3)}},-2);   
\coordinate (J2) at ({5*sqrt(3)/2},-2.5);
\coordinate (K2) at ({3*sqrt(3)},-2);
\coordinate (L2) at ({3*sqrt(3)},-1);

\coordinate (G3) at ({7*sqrt(3)/2},-0.5);
\coordinate (H3) at ({{3*sqrt(3)}},-1);
\coordinate (I3) at ({{3*sqrt(3)}},-2);   
\coordinate (J3) at ({7*sqrt(3)/2},-2.5);
\coordinate (K3) at ({4*sqrt(3)},-2);
\coordinate (L3) at ({4*sqrt(3)},-1);

\coordinate (G4) at ({9*sqrt(3)/2},-0.5);
\coordinate (H4) at ({{4*sqrt(3)}},-1);
\coordinate (I4) at ({{4*sqrt(3)}},-2);   
\coordinate (J4) at ({9*sqrt(3)/2},-2.5);
\coordinate (K4) at ({5*sqrt(3)},-2);
\coordinate (L4) at ({5*sqrt(3)},-1);

\coordinate (G5) at ({11*sqrt(3)/2},-0.5);
\coordinate (H5) at ({{5*sqrt(3)}},-1);
\coordinate (I5) at ({{5*sqrt(3)}},-2);   
\coordinate (J5) at ({11*sqrt(3)/2},-2.5);
\coordinate (K5) at ({6*sqrt(3)},-2);
\coordinate (L5) at ({6*sqrt(3)},-1);

\draw [line width=0.25mm] (G)--(H)--(I)--(J)--(K)--(L)--(G); 
\draw [line width=0.25mm] (A1)--(B1)--(C1)--(D1)--(I1)-- (J1)--(K1)--(L1)--(G1)--(F1)--(A1);
\draw [line width=0.25mm] (B2)--(A2)--(F2)--(A3)--(F3)--(E3)--(D3)--(C3)--(D2)--(C2)--(B2)--(A2); 
\draw [line width=0.25mm] (G2)--(H2)--(I2)--(J2)--(K2)--(L2)--(G2); 
\draw [line width=0.25mm] (G3)--(H3)--(I3)--(J3)--(K3)--(J4)--(K4)--(L4)--(G4)--(H4)--(G3); 
\draw [line width=0.25mm] (A5)--(B5)--(C5)--(D5)--(I5)--(J5)--(K5)--(L5)--(G5)--(F5)--(A5); 
\draw [line width=0.25mm] (A4)--(B4)--(C4)--(D4)--(E4)--(F4)--(A4); 
\draw [line width=0.25mm] (A6)--(B6)--(C6)--(D6)--(E6)--(F6)--(A6); 

\draw[dashed,opacity=0.4] (D)--(E) (D1)--(E1) (E2)--(F2) (K3)--(L3) (D5)--(E5);

\node[] (p) at ({sqrt(3)},0) {2};
\node[] (p) at ({2*sqrt(3)},0) {4};
\node[] (p) at ({3*sqrt(3)},0) {6};
\node[] (p) at ({4*sqrt(3)},0) {8};
\node[] (p) at ({5*sqrt(3)},0) {10};
\node[] (p) at ({6*sqrt(3)},0) {12};
\node[] (p) at ({sqrt(3)/2},-1.5) {1};
\node[] (p) at ({3*sqrt(3)/2},-1.5) {3};
\node[] (p) at ({5*sqrt(3)/2},-1.5) {5};
\node[] (p) at ({7*sqrt(3)/2},-1.5) {7};
\node[] (p) at ({9*sqrt(3)/2},-1.5) {9};
\node[] (p) at ({11*sqrt(3)/2},-1.5) {11};
\end{tikzpicture}  
\caption{A tiling of a honeycomb strip of length 12 using 4 dimers.}
\label{grid1}
\end{figure}
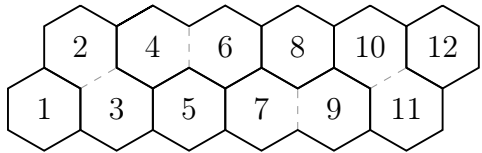 

We start by examining the tilings of such strips by monomers (i.e., single
hexagons) and dimers made of two hexagons joined along an edge. Such
tilings have been considered recently by Dresden and Ziqian \cite{Ziqian},
who found that the total number of such tilings is given by the 
tetranacci numbers. We refine their results in several ways. First, in
Section 2, we obtain formula for the number of such tilings with a
specified number of dimers. Then we consider tilings with colored monomers
and dimers in Section 3. Along the way we obtain combinatorial proofs for
generalizations of several identities involving tetranacci numbers from
the paper by Dresden and Ziqian; we present them in Section 4. Section 5
is devoted to another type of restricted tilings of the honeycomb strip.
There we prohibit horizontal dimers but allow trimers of consecutive
hexagons. The total number of such tilings is given in terms of tribonacci 
numbers, and Padovan and Narayana's cow numbers appear as special cases.
Combinatorial proofs of some related identities are presented in Section 6. 
The paper is closed by some remarks listing some open problems and 
indicating several possible directions for future work.

\section{Tiling a honeycomb strip with exactly $k$ dimers}


In this section we consider a honeycomb strip of length $n$ and its tilings
by hexagonal monomers and dimers shown in Figure \ref{tiles}. We are interested
in the number of such tilings with a given number of dimers. The dimers
can be in any position; in Figure \ref{tiles} we see a descending, a horizontal and
an ascending dimer, from left to right, respectively. Ascending and descending
dimers will be both called slanted when their exact orientation is not
important. We denote the number of all possible tilings of a honeycomb
strip of length $n$ using exactly $k$ dimers by $c_{n,k}$.

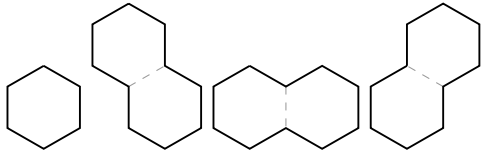
\begin{figure}[h!]
\centering \begin{tikzpicture}[scale=0.55]
\draw [line width=0.25mm] (A) -- (B) -- (C) -- (D) -- (E) -- (F) --(A); 
\end{tikzpicture} \begin{tikzpicture}[scale=0.55]
\draw [line width=0.25mm] (A)--(B)--(C)--(D)--(I)--(J)--(K)--(L)--(G)--(F)--(A); 
\draw [dashed,opacity=0.4] (D)--(E); 
   
\end{tikzpicture} \begin{tikzpicture}[scale=0.55]
\draw [line width=0.25mm] (A)--(B)--(C)--(D)--(E)--(D1)--(E1)--(F1)--(A1)--(B1)--(A); 
\draw [dashed,opacity=0.4] (F)--(E); 
\end{tikzpicture}  \begin{tikzpicture}[scale=0.55]

\draw [line width=0.25mm] (G)--(H)--(I)--(J)--(K)--(L)--(E1)--(F1)--(A1)--(B1)--(G); 
\draw [dashed,opacity=0.4] (L)--(G); 
   
\end{tikzpicture} 
    \caption{Monomer and three possible positions of a dimer tile.}
\label{tiles}
\end{figure} 

Dresden and Ziqian \cite{Ziqian} proved that the total number of all possible
ways to tile a strip with monomers and dimers $h_n$ satisfies recursion
\begin{equation} h_n=h_{n-1}+h_{n-2}+h_{n-3}+h_{n-4}\end{equation} 
with initial values $h_1=1$, $h_2=2$, and $h_3=4$. It makes sense to
define $h_0 = 1$, accounting for the only possible tiling (the empty one)
of the empty honeycomb strip. Their recurrence is the same as the 
recurrence for the tetranacci numbers $Q_n$ (A000078 in \cite{OEIS}) with
shifted initial values. Hence, $h_n=Q_{n+3}$. We wish to determine 
$c_{n,k}$, the number of such tilings using exactly $k$ dimers, and
hence $n-2k$ monomers. It is easy to
see that $c_{n,k}=0$ for $k>\floor*{\frac{n}{2}}$, since the strip with
$n$ hexagons can contain at most $\floor*{\frac{n}{2}}$ dimers. On the lower
end, there is only one tiling without dimers, so $c_{n,0} = 1$ for all $n$.
By stacking $k$ dimers at the beginning of the strip, it is always possible
to tile the remainder by monomers, so it follows that all $c_{n,k}$ for $k$ between $1$ and $\floor*{\frac{n}{2}}$ will be strictly positive. Hence
the numbers $c_{n,k}$ will be arranged in a triangular array without
internal zeros. In
table \ref{Initial Values Table} we give the list of initial values
that can be easily verified. 
\begin{table}[h!]\centering
$\begin{array}{cc}
c_{0,0}=1 \\
c_{1,0}=1 \\
c_{2,0}=1 & c_{2,1}=1 \\
c_{3,0}=1 & c_{3,1}=3 
\end{array}$ 
\caption{Initial values of $c_{n,k}$.}
\label{Initial Values Table}
\end{table} 
 
In the next theorem we give a recurrence relation for $c_{n,k}$. 

\begin{tm} \label{recursive c(n,k) theorem} Let $n \geq 4$ be an integer and
$c_{n,k}$ be the number of ways to tile a honeycomb strip of a length $n$ 
by using exactly $k$ dimers and $n-2k$ monomers. Then the numbers $c_{n,k}$
satisfy the recurrence relation
\begin{equation} 
c_{n,k}=c_{n-1,k}+c_{n-2,k-1}+c_{n-3,k-1}+c_{n-4,k-2} \label{recursive c_{n,k}}
\end{equation}
with the initial conditions given in Table \ref{Initial Values Table}.
\end{tm}
\begin{proof}
We consider an arbitrary tiling which uses $k$ dimers and note that the
$n$-th hexagon can be tiled either by a dimer or by a monomer. The number
off all such tilings with the last hexagon tiled by a monomer is $c_{n-1,k}$,
since the number of dimers $k$ remains the same.
If the last hexagon is a part of a dimer, then we distinguish
two possible situations: either the dimer is slanted or it is horizontal.
The number of tilings ending in a slanted dimer is $c_{n-2,k-1}$, since
the last dimer increases the length of a strip by two and number of
dimers by one. If the dimer is horizontal, it means that it must cover the
$(n-2)$-nd and the $n$-th hexagon. In that case, we have two subcases: either
the $(n-1)$-st hexagon is tiled by monomer, and the rest of the strip
can be tiled in $c_{n-3,k-1}$ ways, or $(n-1)$-st hexagon forms a dimer
with $n-3$-rd hexagon, and rest of the strip can be tiled in $c_{n-4,k-2}$
ways. Described cases are illustrated in Figure \ref{All cases} from
left to right, respectively.
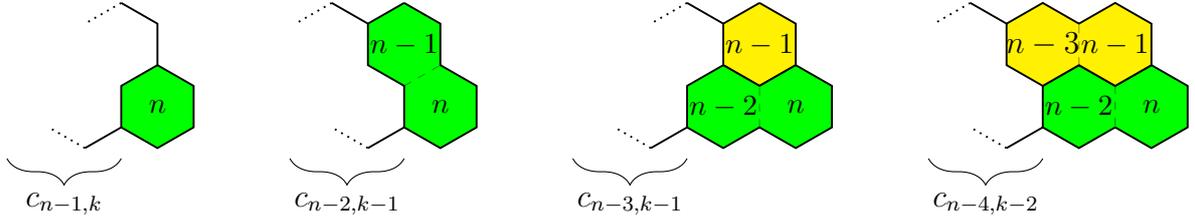
\begin{figure}[h!]
\centering \begin{tikzpicture}[scale=0.55]
\draw [line width=0.25mm,fill=green] (G1)--(H1)--(I1)--(J1)--(K1)--(L1)--(G1) ; 

\draw [line width=0.25mm]  (K)--(J) (E1)--(F1)--(A1); 
\draw [line width=0.25mm, dotted] (A1)--(B1) (J)--(I); 
\draw [decorate,decoration={brace,amplitude=10pt}]
 ({sqrt(3)},-2.75)--(-1,-2.75) node [black,midway,yshift=-0.6cm] 
{$c_{n-1,k}$};
  
\node[font=\small] (p) at ({3*sqrt(3)/2},-1.5) {$n$};
\end{tikzpicture} \hspace{0.9 cm}   \begin{tikzpicture}[scale=0.55]

\draw [line width=0.25mm,fill=green] (A1)--(B1)--(C1)--(D1)--(I1)--(J1)--(K1)--(L1)--(G1)--(F1)--(A1); 

\draw [line width=0.25mm]  (K)--(J) (F)--(A); 
\draw [line width=0.25mm, dotted] (A)--(B) (J)--(I); 

\draw [dashed,opacity=0.4] (D1)--(E1); 
   \draw [decorate,decoration={brace,amplitude=10pt}]
 ({sqrt(3)},-2.75)--(-1,-2.75) node [black,midway,yshift=-0.6cm] 
{$c_{n-2,k-1}$};

\node[font=\small] (p) at ({sqrt(3)},0) {$n-1$};
\node[font=\small] (p) at ({3*sqrt(3)/2},-1.5) {$n$};
\end{tikzpicture} \hspace{0.9 cm} \begin{tikzpicture}[scale=0.55]

\draw [line width=0.25mm, fill=yellow] (A2)--(B2)--(C2)--(D2)--(E2)--(F2)--(A2); 
\draw [line width=0.25mm,fill=green] (G1)--(H1)--(I1)--(J1)--(K1)--(J2)--(K2)--(L2)--(G2)--(H2)--(G1) ; 

\draw [line width=0.25mm]  (K)--(J) (F1)--(A1); 
\draw [line width=0.25mm, dotted] (A1)--(B1) (J)--(I); 

\draw [dashed,opacity=0.4] (K1)--(L1); 
    \draw [decorate,decoration={brace,amplitude=10pt}]
 ({sqrt(3)},-2.75)--(-1,-2.75) node [black,midway,yshift=-0.6cm] 
{$c_{n-3,k-1}$};
\node[font=\small] (p) at ({2*sqrt(3)},0) {$n-1$};
\node[font=\small] (p) at ({3*sqrt(3)/2},-1.5) {$n-2$};
\node[font=\small] (p) at ({5*sqrt(3)/2},-1.5) {$n$};
\end{tikzpicture} \hspace{0.9 cm}  \begin{tikzpicture}[scale=0.55]

\draw [line width=0.25mm,fill=green] (G1)--(H1)--(I1)--(J1)--(K1)--(J2)--(K2)--(L2)--(G2)--(H2)--(G1); 
\draw [line width=0.25mm,fill=yellow] (B1)--(C1)--(D1)--(E1)--(D2)--(E2)--(F2)--(A2)--(B2)--(A1)--(B1); 

\draw [line width=0.25mm]  (K)--(J) (F)--(A); 
\draw [line width=0.25mm, dotted] (A)--(B) (J)--(I); 

\draw [dashed,opacity=0.4] (L1)--(K1); 
\draw [dashed,opacity=0.4] (E1)--(F1); 

  \draw [decorate,decoration={brace,amplitude=10pt}]
 ({sqrt(3)},-2.75)--(-1,-2.75) node [black,midway,yshift=-0.6cm] 
{$c_{n-4,k-2}$};  
\node[] (p) at ({sqrt(3)},0) {$n-3$};
\node[font=\small] (p) at ({2*sqrt(3)},0) {$n-1$};
\node[font=\small] (p) at ({5*sqrt(3)/2},-1.5) {$n$};
\node[font=\small] (p) at ({3*sqrt(3)/2},-1.5) {$n-2$};
\end{tikzpicture} 
\caption{All possible endings of a tiled honeycomb strip of length $n$.}
\label{All cases}
\end{figure}  

Since the listed cases and subcases are disjoint and describe all possible situations,
the total number of tilings is the sum of the respective counting numbers
i.e., $c_{n,k}=c_{n-1,k}+c_{n-2,k-1}+c_{n-3,k-1}+c_{n-4,k-2}$, which proves
our theorem. 
\end{proof}

We are now able to list the initial rows of the triangle of $c_{n,k}$,
which we do in Table \ref{First values c_{n,k}} below.

\begin{table}[h!]
\centering
$\begin{array}{c|ccccccccccc}
n/k & 0 & 1 & 2 & 3 & 4 & 5 & 6 & \cdots& Q_n\\
\hline
0 & 1 &&&&&&&&1\\
1 & 1 &&&&&&&&1\\
2 & 1 & 1 &&&&&&&2\\
3 & 1 & 3& &&&&&&4\\
4 & 1 & 5 & 2 &&&&&&8\\ 
5 & 1 & 7 & 7 &&&&&&15\\
6 & 1 & 9 & 16 & 3 &&&&&29\\
7 & 1 & 11 & 29 & 15 &&&&&56\\
8 & 1 & 13 & 46 & 43 & 5 &&&&108\\
\end{array}$ 
\caption{The initial values of $c_{n,k}$.}
\label{First values c_{n,k}}
\end{table}

The triangle of table \ref{First values c_{n,k}} appears as the sequence
A101350 in \cite{OEIS}. Its leftmost column consists of all 1's, counting the
unique tilings without dimers. The second column seems to be given by
$c_{n,1} = 2n-3$. Indeed, the only dimer in the tiling can cover either
hexagons $(i,i+1)$ for $1 \leq i \leq n-1$ or hexagons $(i,i+2)$ for 
$1 \leq i \leq n-2$, resulting in $2n-3$ possible tilings. 
As expected, the rows of the triangle sum to the (shifted) tetranacci 
numbers, $\sum\limits_{k=0}^{\floor*{\frac{n}{2}}}c_{n,k}=Q_{n+3}$, since
by disregarding values $k$, recurrence \ref{recursive c_{n,k}} becomes the
defining recurrence for the tetranacci numbers. 
The appearance of the Fibonacci numbers as the rightmost diagonal,
$c_{2n,n} = F_{n+1}$, can be readily explained by looking at the inner dual
of the strip. As mentioned before, it is the ladder graph with the descending diagonal in each square, as shown in Figure \ref{Dual} .
Clearly, tilings with $n$ dimers correspond to perfect
\begin{figure}[h!]
\centering

\begin{tikzpicture}[scale=0.55]
\coordinate (S1) at ({sqrt(3)/2},-1.5); 
\coordinate (S2) at ({sqrt(3)},0); 
\coordinate (S3) at ({3*sqrt(3)/2},-1.5); 
\coordinate (S4) at ({2*sqrt(3)},0); 
\coordinate (S5) at ({5*sqrt(3)/2},-1.5); 
\coordinate (S6) at ({3*sqrt(3)},0); 
\coordinate (S7) at ({7*sqrt(3)/2},-1.5); 
\coordinate (S8) at ({4*sqrt(3)},0); 
\coordinate (Q1) at ({8*sqrt(3)},-1.5); 
\coordinate (Q2) at ({8*sqrt(3)},0); 
\coordinate (Q3) at ({8*sqrt(3)+1.5},-1.5); 
\coordinate (Q4) at ({8*sqrt(3)+1.5},0); 
\coordinate (Q5) at ({8*sqrt(3)+3},-1.5); 
\coordinate (Q6) at ({8*sqrt(3)+3},0); 
\coordinate (Q7) at ({8*sqrt(3)+4.5},-1.5); 
\coordinate (Q8) at ({8*sqrt(3)+4.5},0); 
\coordinate (T) at ({6*sqrt(3)},-0.75); 
\coordinate (T1) at ({2.5*sqrt(3)},-3.5); 
\coordinate (T2) at ({8*sqrt(3)+2.25},-3.5); 

\draw (T2)  node {$L_n$};
\draw (T1)  node {$H_{2n}$};
\draw (T)  node {$\longrightarrow$};
\draw [line width=0.25mm] (Q2)--(Q4)--(Q6)--(Q8)--(Q7)--(Q5)--(Q3)--(Q1)--(Q2)--(Q3)--(Q4)--(Q5)--(Q6)--(Q7); 

\draw [line width=0.25mm] (G)--(H)--(I)--(J)--(K)--(L)--(G); 
\draw [line width=0.25mm] (G1)--(H1)--(I1)--(J1)--(K1)--(L1)--(G1); 
\draw [line width=0.25mm] (G2)--(H2)--(I2)--(J2)--(K2)--(L2)--(G2); 
\draw [line width=0.25mm] (G3)--(H3)--(I3)--(J3)--(K3)--(L3)--(G3); 

\draw [line width=0.25mm] (A1)--(B1)--(C1)--(D1)--(E1)-- (F1)--(A1);
\draw [line width=0.25mm] (A2)--(B2)--(C2)--(D2)--(E2)--(F2)--(A2); 
\draw [line width=0.25mm] (A3)--(B3)--(C3)--(D3)--(E3)--(F3)--(A3); 
\draw [line width=0.25mm] (A4)--(B4)--(C4)--(D4)--(E4)--(F4)--(A4); 

\draw [line width=0.25mm] (S2)--(S4)--(S6)--(S8)--(S7)--(S5)--(S3)--(S1)--(S2)--(S3)--(S4)--(S5)--(S6)--(S7); 

\end{tikzpicture}  
\caption{Hexagonal strip of a length $2n$ and its inner dual} \label{Dual}
\end{figure}
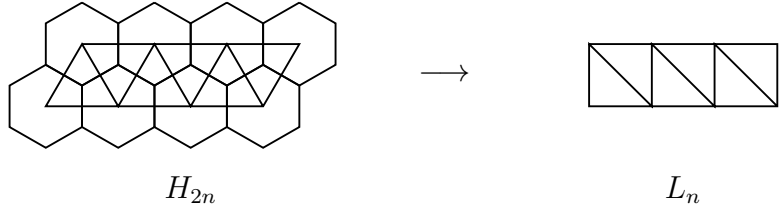 
matchings in the inner dual. A simple parity argument dictates that no
diagonal can participate in such a perfect matching. By omitting the diagonals
we are left with a ladder graph and it is a well known folklore result that
perfect matchings in ladder graphs are counted by Fibonacci numbers. Somewhat
less obvious is the appearance of the convolution of Fibonacci numbers
and shifted Fibonacci numbers as the first descending subdiagonal,
$c_{2n+1,n} = \frac{1}{5}((n+2)F_{n+4} + (n-1)F_{n+2}) = A023610(n)$, but
it follows by observing that the only monomer breaks the strip into two
pieces each of which can be tiled by dimers only, and the number of such
tilings is obtained by summing the corresponding products, hence leading
to convolution. There are no formulas in the OEIS for other columns or
diagonals. In the rest of this section we determine formulas for all
elements of the triangle $c_{n,k}$.

%
%

It is well known that for the Fibonacci numbers one has 
$c_{2n,n}=\sum\limits_{m=0}^n\binom{n-m}{m}$. By writing this as 
$c_{2n,n}=\sum\limits_{m=0}^n\binom{n-m}{m}\binom{n-m}{0}$ and by noting
that a similar formula 
$c_{2n+1,n}=\sum\limits_{m=0}^{n-1}\binom{n-m}{m}\binom{n-m}{1}$ can be
readily verified by induction, it becomes natural to consider
$\sum\limits_{m=0}^{n-k}\binom{n-m}{m}\binom{n-m}{k}$ as the formula
for the elements on descending diagonals. By shifting the indices 
$n\to n-k$ and $k\to n-2k$ one arrives at expression for $c_{n,k}$.

\begin{tm}\label{c_(n,k) formula theorem} The number of ways to tile a
honeycomb strip of length $n$ using $k$ dimers and $n-2k$ monomers
is equal to \begin{equation} 
c_{n,k}=\sum\limits_{m=0}^k\binom{n-k-m}{m}\binom{n-k-m}{n-2k}.
\label{c_(n,k) formula}
\end{equation} 
\end{tm}

Theorem \ref{c_(n,k) formula theorem} can be proven by induction, but
we prefer to present a combinatorial proof. To do that we need to some
new terms and one lemma.

We say that a tiling of a honeycomb strip is {\em breakable} at the position
$k$ if a given tiling can be divided into two tiled strips, first strip
of length $k$ and second of length $n-k$. Note that breaking the strip is only
allowed along the edge of the tile. If no such $k$ exist, we say that
tiling is {\em unbreakable}. 
  
For example, if first two hexagons form a dimer, the tiling is unbreakable
at position $1$, since it is not allowed to break a tiling through the dimer.
As an example, Figure \ref{breakability1} illustrates all breakable positions
of a given tiling.

\begin{figure}[h!]
\centering

\begin{tikzpicture}[scale=0.55]
\coordinate (Fh) at ({sqrt(3)/2},1);
\coordinate (F1h) at ({3*sqrt(3)/2},1);
\coordinate (F2h) at ({5*sqrt(3)/2},1); 
\coordinate (F3h) at ({7*sqrt(3)/2},1);
\coordinate (Kl) at ({sqrt(3)},-2.5);
\coordinate (K1l) at ({2*sqrt(3)},-2.5);
\coordinate (K2l) at ({3*sqrt(3)},-2.5);
\coordinate (K3l) at ({4*sqrt(3)},-2.5);

\draw [line width=0.25mm] (I)--(J)--(K)--(L)--(E1)--(F1)--(A1)--(B1)--(C1)--(H)--(I); 
\draw [line width=0.25mm] (A4)--(B4)--(C4)--(D4)--(E4)--(F4)--(A4);
\draw [line width=0.25mm](A2)--(B2)--(C2)--(D2)--(E2)-- (D3)--(E3)--(F3)--(A3)--(B3)--(A2);
\draw [line width=0.25mm](G2)--(H2)--(I2)--(J2)--(K2)--(J3)--(K3)--(L3)--(G3)--(H3)--(G2); 
\draw [line width=0.25mm] (G1)--(H1)--(I1)--(J1)--(K1)--(L1)--(G1) ; 
\draw[dashed,opacity=0.4] (G)--(L) (E2)--(F2) (K2)--(L2);

\draw [line width=0.5mm](F1h)--(G1)--(H1)--(Kl); 
\draw [line width=0.5mm](F1h)--(G1)--(L1)--(K1l); 
\draw [line width=0.5mm](F3h)--(G3)--(L3)--(K3l); 

\node[] (p) at ({sqrt(3)},0) {2};
\node[] (p) at ({2*sqrt(3)},0) {4};
\node[] (p) at ({3*sqrt(3)},0) {6};
\node[] (p) at ({4*sqrt(3)},0) {8};
\node[] (p) at ({sqrt(3)/2},-1.5) {1};
\node[] (p) at ({3*sqrt(3)/2},-1.5) {3};
\node[] (p) at ({5*sqrt(3)/2},-1.5) {5};
\node[] (p) at ({7*sqrt(3)/2},-1.5) {7};
\end{tikzpicture}  
\caption{Tiling of a honeycomb strip that is breakable at positions $2$,
$3$ and $7$.} \label{breakability1}
\end{figure}
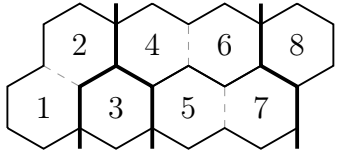

\begin{lm}\label{lemma unbreakable tiling} For $n>4$, every tiled strip
of length $n$ is breakable into four types of unbreakable tiled strips:
length-one strip tiled with a single monomer, length-two strip tiled
with a single dimer, length-three strip tiled with a horizontal dimer and 
a monomer, and length-four strip tiled with two horizontal dimers.   
\end{lm}
\begin{proof} Every left or right slanted dimer forms a
strip of length two. When removed, we are left with smaller strips, each
of them tiled with hexagons and horizontal dimers. Every horizontal dimer
occupies positions in the form $\left\lbrace i,i+2 \right\rbrace$.
If position $i+1$ is occupied by a monomer, hexagons in position $i,i+1$
and $i+2$ form a length-three tiled strip. If position $i+1$ is occupied
by another horizontal dimer, that dimer can occupy positions $i-1,i+1$
or $i+1,i+3$. Either way, those two horizontal dimers form a length-four
tiled strip. After they are removed, we are left with only monomers,
where each monomer forms a simple tiled strip of length one. Those are only
for types of unbreakable tilings. They are illustrated in
Figure  \ref{Unbreakable tiling}.  \end{proof} \begin{figure}[h!]
\centering \begin{tikzpicture}[scale=0.55]
 \draw [line width=0.25mm] (A) -- (B) -- (C) -- (D) -- (E) -- (F) --(A); 
\end{tikzpicture} \hspace{0.9 cm}   \begin{tikzpicture}[scale=0.55]
 \draw [line width=0.25mm] (A1)--(B1)--(C1)--(D1)--(I1)--(J1)--(K1)--(L1)--(G1)--(F1)--(A1); 
\draw [dashed,opacity=0.4] (D1)--(E1);  
\end{tikzpicture} \hspace{0.9 cm} \begin{tikzpicture}[scale=0.55]
\draw [line width=0.25mm] (A2)--(B2)--(C2)--(D2)--(E2)--(F2)--(A2); 
\draw [line width=0.25mm] (G1)--(H1)--(I1)--(J1)--(K1)--(J2)--(K2)--(L2)--(G2)--(H2)--(G1) ; 
\draw [dashed,opacity=0.4] (L1)--(K1);  
\end{tikzpicture} \hspace{0.9 cm}  \begin{tikzpicture}[scale=0.55]
\draw [line width=0.25mm] (G1)--(H1)--(I1)--(J1)--(K1)--(J2)--(K2)--(L2)--(G2)--(H2)--(G1); 
\draw [line width=0.25mm] (B1)--(C1)--(D1)--(E1)--(D2)--(E2)--(F2)--(A2)--(B2)--(A1)--(B1); 
\draw [dashed,opacity=0.4] (L1)--(K1); 
\draw [dashed,opacity=0.4] (E1)--(F1); 
\end{tikzpicture} 
\caption{All unbreakable types of tiled strip. The second and the fourth can be
left or right slanted, and the third can be upside down, depending on the
parity of the position.}\label{Unbreakable tiling}
\end{figure}
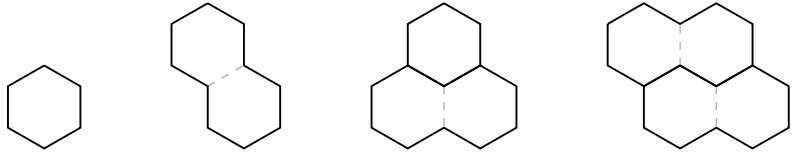   
\begin{proof}[Proof of the Theorem \ref{c_(n,k) formula theorem}] We denote
types of tiled strip from Figure \ref{Unbreakable tiling} by $M$, $D$ ,$T$ and
$V$, from left to right, respectively. By Lemma
\ref{lemma unbreakable tiling}, an arbitrary tiling of
a strip $H_n$ of length $n>4$ can be broken into those four types of unbreakable
tilings. Breaking of a given tiled strip into unbreakable strips produces
unique number of tiled strips of each type. So, let $k_1$ denotes the number
of strips of type $D$, $k_2$ the number of strips of type $T$, $k_3$ the
number of strips of type $V$, and since the strip has length $n$, what remains
are $n-2k_1-3k_2-4k_3$ strips of type $M$. Now we establish
1-1 correspondence between two sets: the first set, that contains all tilings of a strip $H_n$ which by braking produce $k_1$ strips of type $D$,
$k_2$ strips of type $T$, $k_3$ strips of type $V$ and $n-2k_1-3k_2-4k_3$
strips of type $M$, and the second set that contains all permutations with
repetition of a set with $n-k_1-2k_2-3k_3$ elements, where there are $k_1$
elements of type $d$, $k_2$ elements of type $t$, $k_3$ elements of type
$v$ and $n-2k_1-3k_2-4k_3$ elements of type $m$. From an arbitrary permutation we
obtain the corresponding tiling as follows: we replace elements $v$, $t$, $d$ and $m$ with a tiled strips of type $V$, $T$, $D$ and $M$, respectively.  For example, permutation
$dmdvt$ yields the tiling shown in Figure \ref{Tiling dmdvt}. The way to obtain a permutation from a given tiling is obvious.
\begin{figure}[h!]
\centering  \begin{tikzpicture}[scale=0.55]
\draw [line width=0.25mm] (A1)--(B1)--(C1)--(H)--(I)--(J)--(K)--(L)--(E1)--(F1)--(A1); 
\draw [line width=0.25mm] (G1)--(H1)--(I1)--(J1)--(K1)--(L1)--(G1); 
\draw [line width=0.25mm] (A2)--(B2)--(C2)--(H2)--(I2)--(J2)--(K2)--(L2)--(E2)--(F2)--(A2);
\draw [line width=0.25mm] (A3)--(B3)--(C3)--(D3)--(E3)--(D4)--(E4)--(F4)--(A4)--(B4)--(A3); 
\draw [line width=0.25mm] (G3)--(H3)--(I3)--(J3)--(K3)--(J4)--(K4)--(L4)--(G4)--(H4)--(G3); 
\draw [line width=0.25mm] (A5)--(B5)--(C5)--(D5)--(E5)--(D6)--(E6)--(F6)--(A6)--(B6)--(A5); 
\draw [line width=0.25mm] (J5)--(K5)--(L5)--(G5)--(H5)--(I5)--(J5); 
\draw [dashed,opacity=0.4] (C1)--(D1)(D2)--(E2)(B4)--(C4)(H4)--(I4)(E5)--(F5);
\node[] (p1) at (J) [yshift=-3ex] {$d$};
\node[] (p1) at (J1) [yshift=-3.3ex] {$m$};
\node[] (p1) at (J2) [yshift=-3ex] {$d$};
\node[] (p1) at (J3) [xshift=2.8ex,yshift=-3.3ex] {$v$};
\node[] (p1) at (J5) [yshift=-3ex] {$t$};

 \end{tikzpicture} 
\caption{Tiling that corresponds to permutation $dmdvt$.}\label{Tiling dmdvt}
\end{figure}
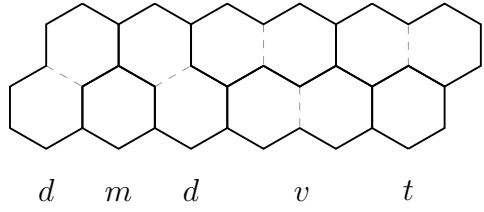

The number of all permutation in the second set is
$\frac{(n-k_1-2k_2-3k_3)!}{k_1!k_2!k_3!(n-2k_1-3k_2-4k_3)!}$. One can
easily verify that this expression can be written as 
$\binom{n-k_1-2k_2-3k_3}{k_1}\binom{n-2k_1-2k_2-3k_3}{k_3}\binom{n-2k_1-2k_2-4k_3}{k_2}$. 
Since we are interested in the number $c_{n,k}$ which denotes the number
of ways to tile a length-$n$ strip that contains exactly $k$ dimers, note
that $k_1+k_2+2k_3$ must be equal to $k$. Hence, we have 
\begin{align*}
c_{n,k}&=\sum\limits_{k_1+k_2+2k_3=k}\binom{n-k_1-2k_2-3k_3}{k_1}\binom{n-2k_1-2k_2-3k_3}{k_3}\binom{n-2k_1-2k_2-4k_3}{k_2}.\end{align*}
By introducing new index of summation $m=k_2+k_3$ and by substitutions
$k_2=m-k_3$ and $k_1=k-k_2-2k_3=k-m-k_3$ we obtain:
\begin{align*}
c_{n,k}&=\sum\limits_{m=0}^k\sum\limits_{k_3=0}^{m}\binom{n-k-m}{k-m-k_3}\binom{n-2k+k_3}{k_3}\binom{n-2k}{m-k_3}\\
&=\sum\limits_{m=0}^k\sum\limits_{k_3=0}^{m}\binom{n-k-m}{n-2k+k_3}\binom{n-2k+k_3}{n-2k}\binom{n-2k}{m-k_3}.
\end{align*}
Finally, by using identity
$\binom{n}{k}\binom{k}{m}=\binom{n}{m}\binom{n-m}{k-m}$ on the first two
binomial coefficients and Vandermonde's convolution
$\sum\limits_{k}\binom{n}{k}\binom{r}{m-k}=\binom{n+r}{m}$ (see \cite{GKP}),
we arrive to:
\begin{align*}
c_{n,k}&=\sum\limits_{m=0}^k\binom{n-k-m}{n-2k}\left(\sum\limits_{k_3=0}^{m}\binom{k-m}{k_3}\binom{n-2k}{m-k_3}\right)\\
&=\sum\limits_{m=0}^k\binom{n-k-m}{n-2k}\binom{n-k-m}{m},
\end{align*} which concludes our proof.
\end{proof}
Since the row sums in the Table \ref{First values c_{n,k}} are tetranacci
numbers, the Theorem \ref{c_(n,k) formula theorem} gives us identity
$$Q_{n+3}=\sum\limits_{k=0}^{\floor*{\frac{n}{2}}}\sum\limits_{m=0}^k\binom{n-k-m}{m}\binom{n-k-m}{n-2k}.$$

\section{Tilings of honeycomb strip with colored dimers and monomers}

Katz and Stenson \cite{KS} used colored squares and dominos to tile
($2\times n$)-rectangular board and obtained a recursive relation for
the number of all ways to tile a board. They also proved some combinatorial
identities involving the number of such tilings.
In this section we do a honeycomb strip analogue. We continue to count
tilings of a hexagon strip with dimers and monomers, but we allow $a$ different
colors for monomers and $b$ different colors for dimers. Let $h^{a,b}_{n}$
denotes the number of all different tilings of a strip with $n$ hexagons.
It is convenient to define $h^{a,b}_{0}=1$. We start with initial values
illustrated  in Figure \ref{Initial values}.  One can easily see that
$h^{a,b}_1=a$ since we have $a$ colors to choose from for a monomer.
Similarly, $h^{a,b}_2=a^2+b$, since we can tile a strip with two monomers
in $a^2$ ways or with one dimer in $b$ ways. For $n=3$, note that if we
use only monomers, we can choose colors in $a^3$ ways, and if we use one
dimer and one monomer, we can put dimer in $3$ different positions and
for each of that positions we can choose colors for tiles in $ab$ ways.
Hence, $h^{a,b}_3=a^3+3ab$.

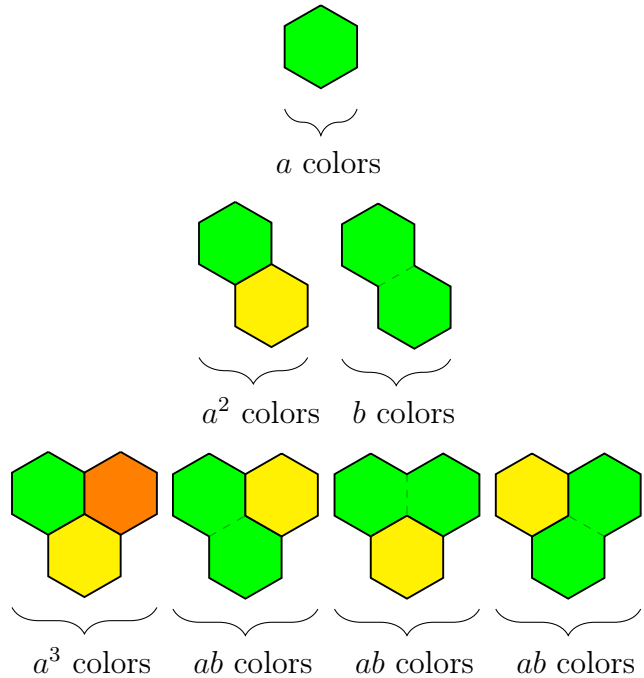
\begin{figure}[h!]
\centering \begin{tikzpicture}[scale=0.55]

\draw [decorate,decoration={brace,amplitude=10pt}]
({sqrt(3)/2},-1.5)--({-sqrt(3)/2},-1.5) node [black,midway,yshift=-0.7cm,xshift=0.1cm] 
{$a$ colors};
\draw [line width=0.25mm, fill=green] (A) -- (B) -- (C) -- (D) -- (E) -- (F) --(A); 
\end{tikzpicture} 
  
\vspace{0.2 cm}

\begin{tikzpicture}[scale=0.55]
\draw [line width=0.25mm, fill=green] (A)--(B)--(C)--(D)--(E)--(F)--(A); 
\draw [line width=0.25mm, fill=yellow] (G)--(H)--(I)--(J)--(K)--(L)--(G); 
\draw [decorate,decoration={brace,amplitude=10pt}]
({sqrt(3)},-2.75)--({-sqrt(3)/2},-2.75) node [black,midway,yshift=-0.7cm,xshift=0.1cm] 
{$a^2$ colors};

\end{tikzpicture} \begin{tikzpicture}[scale=0.55]
\draw [decorate,decoration={brace,amplitude=10pt}]
({sqrt(3)},-2.75)--({-sqrt(3)/2},-2.75) node [black,midway,yshift=-0.7cm,xshift=0.1cm] 
{$b$ colors};

\draw [line width=0.25mm, fill=green] (A)--(B)--(C)--(D)--(I)--(J)--(K)--(L)--(G)--(F)--(A); 

\draw[dashed,opacity=0.4] (D) -- (E); 

\end{tikzpicture} 

\vspace{0.2 cm}

\begin{tikzpicture}[scale=0.55]
\draw [decorate,decoration={brace,amplitude=10pt}]
({3*sqrt(3)/2},-2.75)--({-sqrt(3)/2},-2.75) node [black,midway,yshift=-0.7cm,xshift=0.1cm] 
{$a^3$ colors};

\draw [line width=0.25mm, fill=green] (A)--(B)--(C)--(D)--(E)--(F)--(A); 
\draw [line width=0.25mm, fill=orange] (A1)--(B1)--(C1)--(D1)--(E1)--(F1)--(A1); 
\draw [line width=0.25mm, fill=yellow] (G)--(H)--(I)--(J)--(K)--(L)--(G); 
\end{tikzpicture} \begin{tikzpicture}[scale=0.55]

\draw [decorate,decoration={brace,amplitude=10pt}]
({3*sqrt(3)/2},-2.75)--({-sqrt(3)/2},-2.75) node [black,midway,yshift=-0.7cm,xshift=0.1cm] 
{$ab$ colors};
\draw [line width=0.25mm, fill=green] (A)--(B)--(C)--(D)--(I)--(J)--(K)--(L)--(G)--(F)--(A); 
\draw [line width=0.25mm, fill=yellow] (A1)--(B1)--(C1)--(D1)--(E1)--(F1)--(A1); 

\draw[dashed,opacity=0.4] (D) -- (E); 
\end{tikzpicture} \begin{tikzpicture}[scale=0.55]

\draw [decorate,decoration={brace,amplitude=10pt}]
({3*sqrt(3)/2},-2.75)--({-sqrt(3)/2},-2.75) node [black,midway,yshift=-0.7cm,xshift=0.1cm] 
{$ab$ colors};

\draw [line width=0.25mm, fill=green] (A)--(B)--(C)--(D)--(E)--(D1)--(E1)--(F1)--(A1)--(B1)--(A); 
\draw [line width=0.25mm, fill=yellow] (G)--(H)--(I)--(J)--(K)--(L)--(G); 

\draw[dashed,opacity=0.4] (F) -- (E); 
\end{tikzpicture} \begin{tikzpicture}[scale=0.55]

\draw [decorate,decoration={brace,amplitude=10pt}]
({3*sqrt(3)/2},-2.75)--({-sqrt(3)/2},-2.75) node [black,midway,yshift=-0.7cm,xshift=0.1cm] 
{$ab$ colors};
\draw [line width=0.25mm, fill=green] (C1)--(H)--(I)--(J)--(K)--(D1)--(E1)--(F1)--(A1)--(B1)--(C1); 
\draw [line width=0.25mm, fill=yellow] (A)--(B)--(C)--(D)--(E)--(F)--(A); 
\draw[dashed,opacity=0.4] (D1) -- (C1); 
\end{tikzpicture}
\caption{All possible tilings for $n=1,2,3$.}\label{Initial values}
\end{figure} 
 
In the next theorem we give a recursive relation for $h^{a,b}_n$.
\begin{tm} For $n\geq 4$, the number of all possible tilings of the
honeycomb strip containing $n$ hexagons with $a$ different kinds of
monomer and $b$ different kinds of dimer satisfies the recursive relation
$$h^{a,b}_n=a\cdot h^{a,b}_{n-1}+b\cdot h^{a,b}_{n-2}+ab\cdot h^{a,b}_{n-3}+b^2\cdot h^{a,b}_{n-4}$$
with the initial conditions $h^{a,b}_0 = 1$, $h^{a,b}_1 = a$, $h^{a,b}_2
= a^2+b$, and $h^{a,b}_3 = a^3+ 3ab$.
\end{tm}
\begin{proof} The proof is similar to the proof of Theorem
\ref{recursive c(n,k) theorem}, but here we must also pay attention to
the colors. We consider an arbitrary tiling and note that $n$-th hexagon
can either be tiled by monomer or dimer. In the case when $n$-th
hexagon is tiled by monomer, the rest of the strip can be tiled in
$h^{a,b}_{n-1}$ ways, but the monomer can be colored in
$a$ different ways, which gives us the
total of $a\cdot h^{a,b}_n$ possible ways. If the last hexagon is a part
of a dimer, then we distinguish two possible situations: either the dimer
is slanted,
or the dimer is horizontal. The number of tilings ending in a slanted dimer
is $h^{a,b}_{n-2}$, and the last dimer can be colored in $b$ ways. So there
are $b\cdot h^{a,b}_{n-2}$ such tilings. As in the proof of Theorem
\ref{recursive c(n,k) theorem}, if the dimer is horizontal, it means that
it covers the $(n-2)$-nd and the $n$-th hexagon. In that case, the $(n-1)$-st
hexagon can be tiled by monomer, we can choose colors in $ab$ ways, and the
rest of the strip can be tiled in $h^{a,b}_{n-3}$ ways. This gives us the
$ab\cdot h^{a,b}_{n-3}$ possible tiling in this case. The last case is if the
$(n-1)$-st hexagon forms a dimer with $(n-3)$-rd hexagon. There are
$b^2\cdot h^{a,b}_{n-4}$ such tilings. All cases are illustrated in
Figure \ref{All colored cases}. 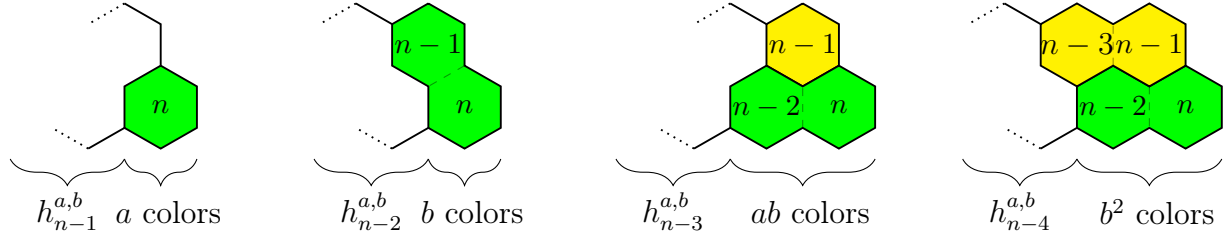
\begin{figure}[h!]
\centering \begin{tikzpicture}[scale=0.55]

\draw [line width=0.25mm,fill=green] (G1)--(H1)--(I1)--(J1)--(K1)--(L1)--(G1) ; 

\draw [line width=0.25mm]  (K)--(J) (E1)--(F1)--(A1); 
\draw [line width=0.25mm, dotted] (A1)--(B1) (J)--(I); 
\draw [decorate,decoration={brace,amplitude=10pt}]
 ({sqrt(3)},-2.75)--(-1,-2.75) node [black,midway,yshift=-0.7cm] 
{$h^{a,b}_{n-1}$}; 
\draw [decorate,decoration={brace,amplitude=10pt}]
({2*sqrt(3)},-2.75)--({sqrt(3)},-2.75) node [black,midway,yshift=-0.7cm,xshift=0.1cm] 
{$a$ colors};

\node[font=\small] (p) at ({3*sqrt(3)/2},-1.5) {$n$};
\end{tikzpicture} \hspace{0.8 cm}   \begin{tikzpicture}[scale=0.55]

\draw [line width=0.25mm,fill=green] (A1)--(B1)--(C1)--(D1)--(I1)--(J1)--(K1)--(L1)--(G1)--(F1)--(A1); 

\draw [line width=0.25mm]  (K)--(J) (F)--(A); 
\draw [line width=0.25mm, dotted] (A)--(B) (J)--(I); 
\draw [dashed,opacity=0.4] (E1)--(D1);

\draw [decorate,decoration={brace,amplitude=10pt}]
 ({sqrt(3)},-2.75)--(-1,-2.75) node [black,midway,yshift=-0.7cm] 
{$h^{a,b}_{n-2}$}; 
\draw [decorate,decoration={brace,amplitude=10pt}]
({2*sqrt(3)},-2.75)--({sqrt(3)},-2.75) node [black,midway,yshift=-0.7cm,xshift=0.1cm] 
{$b$ colors};
\node[font=\small] (p) at ({sqrt(3)},0) {$n-1$};
\node[font=\small] (p) at ({3*sqrt(3)/2},-1.5) {$n$};
\end{tikzpicture} \hspace{0.8 cm} \begin{tikzpicture}[scale=0.55]

\draw [line width=0.25mm, fill=yellow] (A2)--(B2)--(C2)--(D2)--(E2)--(F2)--(A2); 
\draw [line width=0.25mm,fill=green] (G1)--(H1)--(I1)--(J1)--(K1)--(J2)--(K2)--(L2)--(G2)--(H2)--(G1) ; 

\draw [line width=0.25mm]  (K)--(J) (F1)--(A1); 
\draw [line width=0.25mm, dotted] (A1)--(B1) (J)--(I); 

\draw [dashed,opacity=0.4] (K1)--(L1); 
    \draw [decorate,decoration={brace,amplitude=10pt}]
 ({sqrt(3)},-2.75)--(-1,-2.75) node [black,midway,yshift=-0.7cm] 
{$h^{a,b}_{n-3}$}; 
\draw [decorate,decoration={brace,amplitude=10pt}]
({3*sqrt(3)},-2.75)--({sqrt(3)},-2.75) node [black,midway,yshift=-0.7cm,xshift=0.1cm] 
{$ab$ colors};
\node[font=\small] (p) at ({2*sqrt(3)},0) {$n-1$};
\node[font=\small] (p) at ({3*sqrt(3)/2},-1.5) {$n-2$};
\node[font=\small] (p) at ({5*sqrt(3)/2},-1.5) {$n$};
\end{tikzpicture} \hspace{0.8 cm}  \begin{tikzpicture}[scale=0.55]

\draw [line width=0.25mm,fill=green] (G1)--(H1)--(I1)--(J1)--(K1)--(J2)--(K2)--(L2)--(G2)--(H2)--(G1); 
\draw [line width=0.25mm,fill=yellow] (B1)--(C1)--(D1)--(E1)--(D2)--(E2)--(F2)--(A2)--(B2)--(A1)--(B1); 

\draw [line width=0.25mm]  (K)--(J) (F)--(A); 
\draw [line width=0.25mm, dotted] (A)--(B) (J)--(I); 

\draw [dashed,opacity=0.4] (L1)--(K1); 
\draw [dashed,opacity=0.4] (E1)--(F1); 

 \draw [decorate,decoration={brace,amplitude=10pt}]
 ({sqrt(3)},-2.75)--(-1,-2.75) node [black,midway,yshift=-0.7cm] 
{$h^{a,b}_{n-4}$}; 
\draw [decorate,decoration={brace,amplitude=10pt}]
({3*sqrt(3)},-2.75)--({sqrt(3)},-2.75) node [black,midway,yshift=-0.7cm,xshift=0.1cm] 
{$b^2$ colors};
\node[] (p) at ({sqrt(3)},0) {$n-3$};
\node[font=\small] (p) at ({2*sqrt(3)},0) {$n-1$};
\node[font=\small] (p) at ({5*sqrt(3)/2},-1.5) {$n$};
\node[font=\small] (p) at ({3*sqrt(3)/2},-1.5) {$n-2$};
\end{tikzpicture}  
\caption{All possible endings of a colored tiling of a strip with $n$ hexagons.}
\label{All colored cases}
\end{figure}  
This gives us relation
$h^{a,b}_n=a\cdot h^{a,b}_{n-1}+b\cdot h^{a,b}_{n-2}+ab\cdot h^{a,b}_{n-3}+b^2\cdot h^{a,b}_{n-4}$,
which proves our theorem. 
\end{proof}

We can now list some first values of $h^{a,b}_n$. 
\begin{table}[h!]
\centering
$\begin{array}{r|l}
n & h^{a,b}_{n}\\
\hline
0 & 1\\
1 & a\\
2 & a^2+b\\
3 & a^3+3ab\\
4 & a^4+5a^2b+2b^2\\ 
5 & a^5+7a^3b+7a b^2\\
6 & a^6+9a^4b+16a^2b^2+3b^3
\end{array}$  \caption{Some first values of $h^{a,b}_{n}$.}\label{First values h_{a,b,n}}
\end{table}
We can notice that the values $c_{n,k}$ from the last section appear in
every row as coefficients of a bivariate polynomial. Connection between
these values is given in the next theorem:

\begin{tm} The number $h^{a,b}_{n}$ of all possible tilings of the
honeycomb strip of length $n$ with monomers of $a$ different colors and
dimers of $b$ different colors is given by
$$h^{a,b}_n=\sum\limits_{k=0}^{\floor*{\frac{n}{2}}}c_{n,k}a^{n-2k}b^k=\sum\limits_{k=0}^{\floor*{\frac{n}{2}}}\sum\limits_{m=0}^k\binom{n-k-m}{m}\binom{n-k-m}{k-m}a^{n-2k}b^k.$$
\end{tm}

\begin{proof}
We could prove the theorem by induction, but again we present a simple 
combinatorial
proof. The number $h^{a,b}_n$ denotes the number of all possible tilings
of the strip of a length $n$. For a fixed $0\leq k\leq \floor{\frac{n}{2}}$,
there are $c_{n,k}$ possible ways to tile a strip with exactly $k$ dimers,
and since this tiling have $k$ dimers and $n-2k$ monomers, the colors can
be selected in $a^{n-2k}b^k$ ways which gives a total of $c_{n,k}a^{n-2k}b^k$
possible tilings. Since every tiling of the strip can contain
$0$, $1$, ... , $\floor*{\frac{n}{2}}-1$ or $\floor*{\frac{n}{2}}$ dimers,
the overall number of tilings is the sum of these cases, that is
$h^{a,b}_n=\sum\limits_{k=0}^{\floor*{\frac{n}{2}}}c_{n,k}a^{n-2k}b^k$.  
\end{proof}
           
\section{Some (generalized) combinatorial identities involving tetranacci numbers}

In this section we generalize several of the identities obtained by
Dresden and Ziqian \cite{Ziqian} to the case of colored tilings
of a honeycomb strip. All of the following identities reduce to the
mentioned identities of Dresden and Ziqian by setting $a=b=1$.   
        
\begin{tm} For every $m,n\geq 0$  \begin{dmath*}
h^{a,b}_{m+n}=h^{a,b}_{m}h^{a,b}_{n}+ h^{a,b}_{m-1}\left(bh^{a,b}_{n-1}+abh^{a,b}_{n-2}+b^2h^{a,b}_{n-3}\right)+ h^{a,b}_{m-2}\left(ab h^{a,b}_{n-1}+b^2h^{a,b}_{n-2}\right)+b^2 h^{a,b}_{n-1}h^{a,b}_{m-3}.
\end{dmath*}  \end{tm}  
 \begin{proof} 
We consider a tiling of a honeycomb strip containing $m+n$ hexagons. We
have $h^{a,b}_{m+n}$ such tilings. On the other hand, there are
$h^{a,b}_{m}\cdot h^{a,b}_{n}$ tilings that are breakable at position $m$,
as shown in the Figure \ref{breakable at m}.
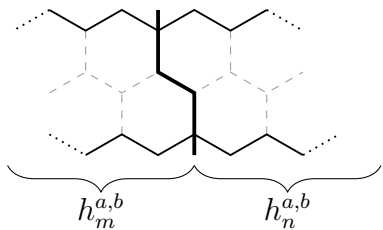
\begin{figure}[h!]
\centering
\begin{tikzpicture}[scale=0.55]
\coordinate (F1h) at ({3*sqrt(3)/2},1); 

\draw[dashed,opacity=0.4] (D)--(E)--(D1)--(E1)--(D2)--(E2)--(D3)--(E3) (F)--(E) (F2)--(E2) (K)--(L) (K2)--(L2); 
 
\draw [line width=0.25mm](A)--(F)--(A1)--(F1)--(A2)--(F2)--(A3) (J)--(K)--(J1)--(K1)--(J2)--(K2)--(J3) ; 
\draw[line width=0.25mm,dotted] (A)--(B) (A3)--(F3) (I)--(J) (J3)--(K3); 

\draw [line width=0.5mm](F1h)--(F1)--(E1)--(L1)--(K1)--(K1l); 
\draw [decorate,decoration={brace,amplitude=10pt}]
 ({2*sqrt(3)},-2.75)--(-1,-2.75) node [black,midway,yshift=-0.6cm] 
{$h^{a,b}_{m}$};
\draw [decorate,decoration={brace,amplitude=10pt}]
 ({4*sqrt(3)+1},-2.75)--({2*sqrt(3)},-2.75) node [black,midway,yshift=-0.6cm] 
{$h^{a,b}_{n}$};
\end{tikzpicture}  
\caption{Breakable tiling at position $m$.} \label{breakable at m}
\end{figure} 
All other tilings are unbreakable at position $m$. If that is the case,
unbreakability can occur because of the right-inclined, left-inclined or
horizontal dimer crossing the line of the break. Figure
\ref{not breakable at m} shows all possible situations that can occur
if tiling is not breakable at position $m$. Note that any tiling of a
honeycomb strip is breakable if $n>4$. 
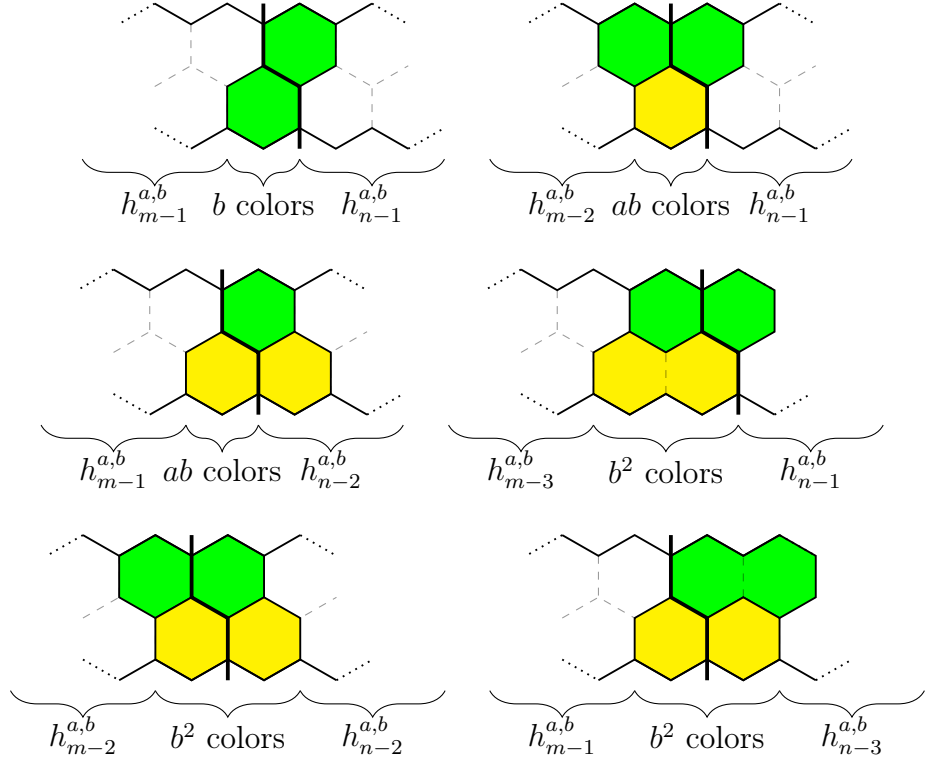
\begin{figure}[h!]
\centering
\begin{tikzpicture}[scale=0.55]
\draw[dashed,opacity=0.4] (D)--(E)--(D1)--(E1)--(D2)--(E2)--(D3)--(E3) (F)--(E) (F2)--(E2) (K)--(L) (K2)--(L2); 
\draw[line width=0.25mm, fill=green] (A2)--(B2)--(C2)--(H1)--(I1)--(J1)--(K1)--(L1)--(E2)--(F2)--(A2);

\draw [line width=0.25mm](A)--(F)--(A1)--(F1)--(A2)--(F2)--(A3) (J)--(K)--(J1)--(K1)--(J2)--(K2)--(J3) ; 
\draw[line width=0.25mm,dotted] (A)--(B) (A3)--(F3) (I)--(J) (J3)--(K3); 

\draw [line width=0.5mm](F1h)--(F1)--(E1)--(L1)--(K1)--(K1l); 
\draw [decorate,decoration={brace,amplitude=10pt}]
 ({sqrt(3)},-2.75)--({-sqrt(3)},-2.75) node [black,midway,yshift=-0.6cm] 
{$h^{a,b}_{m-1}$};
\draw [decorate,decoration={brace,amplitude=10pt}]
 ({4*sqrt(3)},-2.75)--({2*sqrt(3)},-2.75) node [black,midway,yshift=-0.6cm] 
{$h^{a,b}_{n-1}$};\draw [decorate,decoration={brace,amplitude=10pt}]
 ({2*sqrt(3)},-2.75)--({sqrt(3)},-2.75) node [black,midway,yshift=-0.6cm] 
{$b$ colors};
\end{tikzpicture} \hspace{0.3 cm}  \begin{tikzpicture}[scale=0.55]
\draw[dashed,opacity=0.4] (D)--(E)--(D1)--(E1)--(D2)--(E2)--(D3)--(E3) (F)--(E) (F2)--(E2) (K)--(L) (K2)--(L2); 
\draw[line width=0.25mm, fill=green] (B1)--(A1)--(F1)--(A2)--(F2)--(E2)--(D2)--(C2)--(D1)--(C1)--(B1); 
\draw[line width=0.25mm, fill=yellow] (G1)--(H1)--(I1)--(J1)--(K1)--(L1)--(G1);

\draw [line width=0.25mm](A)--(F)--(A1)--(F1)--(A2)--(F2)--(A3) (J)--(K)--(J1)--(K1)--(J2)--(K2)--(J3) ; 
\draw[line width=0.25mm,dotted] (A)--(B) (A3)--(F3) (I)--(J) (J3)--(K3); 

\draw [line width=0.5mm](F1h)--(F1)--(E1)--(L1)--(K1)--(K1l); 
\draw [decorate,decoration={brace,amplitude=10pt}]
 ({sqrt(3)},-2.75)--({-sqrt(3)},-2.75) node [black,midway,yshift=-0.6cm] 
{$h^{a,b}_{m-2}$};
\draw [decorate,decoration={brace,amplitude=10pt}]
 ({4*sqrt(3)},-2.75)--({2*sqrt(3)},-2.75) node [black,midway,yshift=-0.6cm] 
{$h^{a,b}_{n-1}$};\draw [decorate,decoration={brace,amplitude=10pt}]
 ({2*sqrt(3)},-2.75)--({sqrt(3)},-2.75) node [black,midway,yshift=-0.6cm] 
{$ab$ colors};
\end{tikzpicture}  

\vspace{0.4cm} 

\begin{tikzpicture}[scale=0.55]

\draw[dashed,opacity=0.4] (D)--(E)--(D1)--(E1)--(D2)--(E2)--(D3)--(E3) (F)--(E) (F2)--(E2) (K)--(L) (K2)--(L2); 
\draw[line width=0.25mm, fill=green] (B2)--(A2)--(F2)--(E2)--(D2)--(C2)--(B2); 
\draw[line width=0.25mm, fill=yellow] (G1)--(H1)--(I1)--(J1)--(K1)--(J2)--(K2)--(L2)--(G2)--(H2)--(G1);

\draw [line width=0.25mm](A)--(F)--(A1)--(F1)--(A2)--(F2)--(A3) (J)--(K)--(J1)--(K1)--(J2)--(K2)--(J3) ; 
\draw[line width=0.25mm,dotted] (A)--(B) (A3)--(F3) (I)--(J) (J3)--(K3); 

\draw [line width=0.5mm](F1h)--(F1)--(E1)--(L1)--(K1)--(K1l); 
\draw [decorate,decoration={brace,amplitude=10pt}]
 ({sqrt(3)},-2.75)--({-sqrt(3)},-2.75) node [black,midway,yshift=-0.6cm] 
{$h^{a,b}_{m-1}$};
\draw [decorate,decoration={brace,amplitude=10pt}]
 ({4*sqrt(3)},-2.75)--({2*sqrt(3)},-2.75) node [black,midway,yshift=-0.6cm] 
{$h^{a,b}_{n-2}$};\draw [decorate,decoration={brace,amplitude=10pt}]
 ({2*sqrt(3)},-2.75)--({sqrt(3)},-2.75) node [black,midway,yshift=-0.6cm] 
{$ab$ colors};
\end{tikzpicture}  \hspace{0.3 cm} \begin{tikzpicture}[scale=0.55]

\draw[dashed,opacity=0.4] (D)--(E)--(D1)--(E1)--(D2)--(E2)--(D3)--(E3) (F)--(E) (F2)--(E2) (K)--(L) (K2)--(L2) (K)--(L); 
\draw[line width=0.25mm, fill=green] (B3)--(A3)--(F3)--(E3)--(D3)--(C3)--(D2)--(C2)--(B2)--(A2)--(B3); 
\draw[line width=0.25mm, fill=yellow] (G1)--(H1)--(I1)--(J1)--(K1)--(J2)--(K2)--(L2)--(G2)--(H2)--(G1); 
\draw[dashed,opacity=0.4] (K1)--(L1);

\draw [line width=0.25mm](A)--(F)--(A1)--(F1)--(A2)--(F2)--(A3) (J)--(K)--(J1)--(K1)--(J2)--(K2)--(J3) ; 
\draw[line width=0.25mm,dotted]   (A)--(B) (A3)--(F3) (I)--(J) (J3)--(K3); 

\draw [line width=0.5mm](F2h)--(F2)--(E2)--(L2)--(K2)--(K2l); 
\draw [decorate,decoration={brace,amplitude=10pt}]
 ({sqrt(3)},-2.75)--({-sqrt(3)},-2.75) node [black,midway,yshift=-0.6cm] 
{$h^{a,b}_{m-3}$};
\draw [decorate,decoration={brace,amplitude=10pt}]
 ({5*sqrt(3)},-2.75)--({3*sqrt(3)},-2.75) node [black,midway,yshift=-0.6cm] 
{$h^{a,b}_{n-1}$};\draw [decorate,decoration={brace,amplitude=10pt}]
 ({3*sqrt(3)},-2.75)--({sqrt(3)},-2.75) node [black,midway,yshift=-0.6cm] 
{$b^2$ colors};
\end{tikzpicture} \vspace{0.4 cm}

 \begin{tikzpicture}[scale=0.55]

\draw[dashed,opacity=0.4] (D)--(E)--(D1)--(E1)--(D2)--(E2)--(D3)--(E3) (F)--(E) (F2)--(E2) (K)--(L) (K2)--(L2); 
\draw[line width=0.25mm, fill=green] (B2)--(A2)--(F2)--(E2)--(D2)--(C2)--(D1)--(C1)--(B1)--(A1)--(B2); 
\draw[line width=0.25mm, fill=yellow] (G1)--(H1)--(I1)--(J1)--(K1)--(J2)--(K2)--(L2)--(G2)--(H2)--(G1);

\draw [line width=0.25mm](A)--(F)--(A1)--(F1)--(A2)--(F2)--(A3) (J)--(K)--(J1)--(K1)--(J2)--(K2)--(J3) ; 
\draw[line width=0.25mm,dotted] (A)--(B) (A3)--(F3) (I)--(J) (J3)--(K3); 

\draw [line width=0.5mm](F1h)--(F1)--(E1)--(L1)--(K1)--(K1l); 
\draw [decorate,decoration={brace,amplitude=10pt}]
 ({sqrt(3)},-2.75)--({-sqrt(3)},-2.75) node [black,midway,yshift=-0.6cm] 
{$h^{a,b}_{m-2}$};
\draw [decorate,decoration={brace,amplitude=10pt}]
 ({5*sqrt(3)},-2.75)--({3*sqrt(3)},-2.75) node [black,midway,yshift=-0.6cm] 
{$h^{a,b}_{n-2}$};\draw [decorate,decoration={brace,amplitude=10pt}]
 ({3*sqrt(3)},-2.75)--({sqrt(3)},-2.75) node [black,midway,yshift=-0.6cm] 
{$b^2$ colors};
\end{tikzpicture} \hspace{0.3 cm} \begin{tikzpicture}[scale=0.55]

\draw[dashed,opacity=0.4] (D)--(E)--(D1)--(E1)--(D2)--(E2)--(D3)--(E3) (F)--(E) (F2)--(E2) (K)--(L) (K2)--(L2); 
\draw[line width=0.25mm, fill=green] (B2)--(A2)--(B3)--(A3)--(F3)--(E3)--(D3)--(C3)--(D2)--(C2)--(B2); 
\draw[line width=0.25mm, fill=yellow] (G1)--(H1)--(I1)--(J1)--(K1)--(J2)--(K2)--(L2)--(G2)--(H2)--(G1);

\draw [line width=0.25mm](A)--(F)--(A1)--(F1)--(A2)--(F2)--(A3) (J)--(K)--(J1)--(K1)--(J2)--(K2)--(J3) ; 
\draw[line width=0.25mm,dotted] (A)--(B) (A3)--(F3) (I)--(J) (J3)--(K3); 
\draw[dashed,opacity=0.4] (E2)--(F2); 

\draw [line width=0.5mm](F1h)--(F1)--(E1)--(L1)--(K1)--(K1l); 
\draw [decorate,decoration={brace,amplitude=10pt}]
 ({sqrt(3)},-2.75)--({-sqrt(3)},-2.75) node [black,midway,yshift=-0.6cm] 
{$h^{a,b}_{m-1}$};
\draw [decorate,decoration={brace,amplitude=10pt}]
 ({5*sqrt(3)},-2.75)--({3*sqrt(3)},-2.75) node [black,midway,yshift=-0.6cm] 
{$h^{a,b}_{n-3}$};\draw [decorate,decoration={brace,amplitude=10pt}]
 ({3*sqrt(3)},-2.75)--({sqrt(3)},-2.75) node [black,midway,yshift=-0.6cm] 
{$b^2$ colors};
\end{tikzpicture}
\caption{Layouts that can occur if tiling is not breakable at position $m$.} \label{not breakable at m}
\end{figure} 
Summing all these cases gives us the proof of the theorem. \end{proof}

Our second identity counts tilings of the strip containing at least one dimer.

\begin{tm} For every integer $n\geq 1$,
\begin{equation}
h^{a,b}_{n}-a^n=b h^{a,b}_{n-2}+2b\sum\limits_{k=3}^n a^{k-2} h^{a,b}_{n-k}+b^2\sum\limits_{k=3}^n a^{k-3} h^{a,b}_{n-k-1}.
\end{equation}
\end{tm}        
\begin{proof} We prove the result by double counting of all ways of to tile
a strip in which there is at least one dimer. On one hand, there are
$h^{a,b}_n-a^n$ such tilings, since
the only tiling without dimers uses only monomers, and we can choose
colors in $a^n$ ways. The other way to count such tilings keeps trace of the
position where the first dimer occurs. Since the dimer covers two positions
in the strip, we use the larger number to determine it position. For
example, dimer occupying hexagons $1$ (or $2$) and $3$ has a position $3$.
First we start with slanted dimers. If the position of
the first dimer is $k$ for $k\geq 2$, the first part of the strip consists
of $k-2$ monomers and the rest of the strip can be tiled in $h^{a,b}_{n-k}$
ways, which gives us total of $b\sum\limits_{k=2}^na^{k-2}h^{a,b}_{n-k}$
ways. We must now consider the horizontal dimer case. Note that the
horizontal dimer cannot have positions $1$ and $2$. If the position of dimer
is $k$ for $k\geq 3$, then the dimer occupies hexagons $k-2$ and $k$. We have
two subcases, depending on whether the $(k-1)$-st hexagon is tiled by a
monomer or by a dimer. In the second
case, it must be paired with $(k+1)$-st hexagon, since position $k$ is first
to occur. In the first subcase, dimer and monomer can be colored in $ab$
ways, the first part of the strip consisting of $k-3$ monomers can be
colored in $a^{k-3}$ ways, and the rest of the strip can be tiled in
$h^{a,b}_{n-k}$ ways, which gives us the total of
$b\sum\limits_{k=3}^na^{k-2}h^{a,b}_{n-k}$ ways. The latter subcase
involves two dimers, the first occupying hexagons $k-2$ and $k$, and the
second covering $k-1$ and $k+1$. These dimers can be colored in $b^2$ ways, the
first part of the strip consisting of $k-3$ monomers can be colored in
$a^{k-3}$ ways, and the rest can be tiled in $h^{a,b}_{n-k-1}$ ways. Since
all the cases are disjoint, the overall number is the sum of the respective
counting numbers, which proves our theorem.
 \end{proof}
 
We conclude this section with a pair of identities counting tilings of
the strip containing at least one monomer. 

\begin{tm}For every integer $n\geq 0$ we have \begin{equation}
h^{a,b}_{2n}-b^nF_{n+1}=a\sum\limits_{k=0}^n   b^k h^{a,b}_{2n-2k-1}F_{k+2}\end{equation}
and
\begin{equation}
h^{a,b}_{2n-1}=a\sum\limits_{k=0}^n   b^k h^{a,b}_{n-2k-1}F_{k+2}.\end{equation} \end{tm}        
\begin{proof} The number of tilings of the $2n$-strip containing only
dimers is $b^nF_{n+1}$. Hence, the number of tilings containing at least
one monomer is $h^{a,b}_{n}-b^nF_{n+1}$.  On the other hand, we can count
such tilings based on the position of the first monomer. First we consider
the odd positions in the strip. If the first monomer occurs at position $2k+1$,
for some $0\leq k\leq n-1$, the first part of the strip is tiled only by dimers,
and that can be done in $b^kF_{k+1}$ ways, the monomer can be colored in
$a$ ways, and the rest of the strip can be tiled in $h^{a,b}_{2n-2k-1}$ ways. 
Figure \ref{First hexagon, k odd}  illustrates this case. 
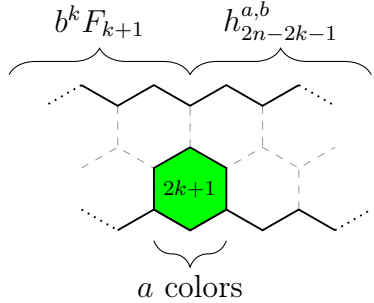
\begin{figure}[h!]
\centering

\begin{tikzpicture}[scale=0.55]

\draw [line width=0.25mm] (A3)--(B3)--(A2)--(B2)--(A1)--(B1)--(A) (J3)--(K2)--(J2)--(K1)--(J1)--(K)--(J); 

\draw [line width=0.25mm,fill=green] (G1)--(H1)--(I1)--(J1)--(K1)--(L1)--(G1); 
\draw[dashed,opacity=0.4] (D)--(E)--(D1)--(E1)--(D2)--(E2)--(D3)--(E3) (F)--(E) (F2)--(E2) (K)--(L) (K2)--(L2) (K1)--(L1) (E1)--(F1); 
\draw[line width=0.25mm,dotted] (A)--(B) (A3)--(F3) (I)--(J) (J3)--(K3);

\draw [decorate,decoration={brace,amplitude=10pt}]
 ({-sqrt(3)},1.25)--({3*sqrt(3)/2},1.25) node [black,midway,yshift=0.7cm] 
{$b^k F_{k+1}$};
\draw [decorate,decoration={brace,amplitude=10pt}]
 ({3*sqrt(3)/2},1.25)--({8*sqrt(3)/2},1.25) node [black,midway,yshift=0.7cm] 
{$h^{a,b}_{2n-2k-1}$};\draw [decorate,decoration={brace,amplitude=10pt}]
 ({2*sqrt(3)},-2.75)--({sqrt(3)},-2.75) node [black,midway,yshift=-0.6cm] 
{$a$ colors};
\node[] (p) at ({3*sqrt(3)/2},-1.5) {$_{2k+1}$};
\end{tikzpicture}  
\caption{The hexagon occurs at position $2k+1$.} \label{First hexagon, k odd}
\end{figure}  

Since the monomer can occur at any position $2k+1$ for $0\leq k \leq n-1$,
the total number of ways that monomer occurs at odd position is
$a\sum\limits_{k=0}^{n-1}b^kh^{a,b}_{2n-2k-1}F_{k+1}$. 

Now we consider the even positions. The case is similar, but there are
some different details. If the first monomer occurs at position $2k$ for
$1\leq k \leq n$, then all $2k-1$ hexagons must be tiled with dimers. For this
to be possible, the $(2k+1)$-st hexagon must be tiled by the same dimer as
$(2k-1)$-st. This dimer and monomer can be colored in $ab$ ways. The first
part of the strip containing $2k-2$ hexagons can be tiled only by dimers and in
$b^{k-1}F_{k}$ ways, and the rest of the strip in $h^{a,b}_{2n-2k-1}$ ways.
This case is illustrated in Figure \ref{First hexagon, k even}. 

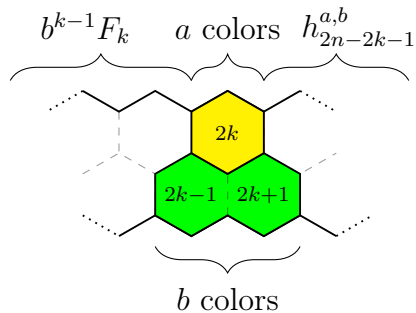
\begin{figure}[h!]
\centering

\begin{tikzpicture}[scale=0.55]

\draw [line width=0.25mm] (A3)--(B3)--(A2)--(B2)--(A1)--(B1)--(A) (J3)--(K2)--(J2)--(K1)--(J1)--(K)--(J); 

\draw [line width=0.25mm,fill=green]  (G1)--(H1)--(I1)--(J1)--(K1)--(J2)--(K2)--(L2)--(G2)--(H2)--(G1) ; 
\draw [line width=0.25mm,fill=yellow]  (A2)--(B2)--(C2)--(D2)--(E2)--(F2)--(A2); 

\draw[dashed,opacity=0.4] (D)--(E)--(D1)--(E1)--(D2)--(E2)--(D3)--(E3) (F)--(E) (F2)--(E2) (K)--(L) (K2)--(L2) (K1)--(L1) ; 
\draw[line width=0.25mm,dotted] (A)--(B) (A3)--(F3) (I)--(J) (J3)--(K3);

\draw [decorate,decoration={brace,amplitude=10pt}]
 ({-sqrt(3)},1.25)--({3*sqrt(3)/2},1.25) node [black,midway,yshift=0.7cm,xshift=-0.2cm] 
{$b^{k-1} F_{k}$};
\draw [decorate,decoration={brace,amplitude=10pt}]
 ({5*sqrt(3)/2},1.25)--({9*sqrt(3)/2},1.25) node [black,midway,yshift=0.7cm,xshift=0.3cm] 
{$h^{a,b}_{2n-2k-1}$};
\draw [decorate,decoration={brace,amplitude=10pt}] ({3*sqrt(3)/2},1.25)--({5*sqrt(3)/2},1.25) node [black,midway,yshift=0.7cm] 
{$a$ colors};\draw [decorate,decoration={brace,amplitude=10pt}]
 ({3*sqrt(3)},-2.75)--({1*sqrt(3)},-2.75) node [black,midway,yshift=-0.7cm] 
{$b$ colors};

\node[] (p) at ({2*sqrt(3)},0) {$_{2k}$};
\node[] (p) at ({5*sqrt(3)/2},-1.5) {$_{2k+1}$};
\node[] (p) at ({3*sqrt(3)/2},-1.5) {$_{2k-1}$};
\end{tikzpicture}  
\caption{The hexagon occurs at position $2k$.} \label{First hexagon, k even}
\end{figure}
The number of tilings where monomer occurs at even position is $ab\sum\limits_{k=0}^{n-1}b^{k-1}h^{a,b}_{2n-2k-1}F_{k}$, and the total number of tilings is the sum of these two cases. Since $h^{a,b}_{-1}=0$ and $F_0=0$, the first sum can be extended to $k=n$ and the second to $k=0$. 
\begin{align*}
h^{a,b}_{2n}-b^nF_{n+1}&=a\sum\limits_{k=0}^{n-1}b^kh^{a,b}_{2n-2k-1}F_{k+1}+ab\sum\limits_{k=1}^{n}b^{k-1}h^{a,b}_{2n-2k-1}F_{k}\\
&=a\sum\limits_{k=0}^{n}b^kh^{a,b}_{2n-2k-1}F_{k+1}+a\sum\limits_{k=0}^{n}b^{k}h^{a,b}_{2n-2k-1}F_{k}\\
&=a\sum\limits_{k=0}^{n}b^kh^{a,b}_{2n-2k-1}F_{k+2}
\end{align*}

The proof of second identity is similar. When the length of the strip is
odd, i.e. $2n-1$, the left hand side is $h^{a,b}_n$, since it cannot be
tiled only by dimers, and the proof for the right hand side is the same,
hence the theorem follows.

\end{proof}

\section{Tiling of a honeycomb strip and tribonacci numbers}

The tribonacci numbers (sequence A000073 in OEIS \cite{OEIS}) are the sequence
of integers starting with $T_0=0$, $T_1=0$ and $T_2=1$ and defined by
recursive relation
\begin{equation}\label{tribonacci recurence}
T_n=T_{n-1}+T_{n-2}+T_{n-3},\text{ for }n\geq 3.
\end{equation}
For the reader's convenience we list a first few values of the sequence
in Table \ref{tribonacci first values}.

\begin{table}[h!]\centering
$\begin{array}{c|ccccccccccc}
n & 0 & 1 & 2& 3& 4& 5& 6& 7& 8& 9& 10\\
\hline
T_n & 0& 0&  1&  1&  2&  4&  7&  13&  27&  44&  81 
\end{array}$ \caption{The first few values of tribonacci numbers.}
\label{tribonacci first values}
\end{table}

In this section we are still interested in counting all tilings of a
honeycomb strip of a given length, but now by using different types of
tiles. We still allow monomers and slanted dimers, but we prohibit horizontal
dimers. In addition, we allow trimers of consecutively numbered hexagons. By
prohibiting horizontal dimers we effectively suppress longer-range connections
represented by horizontal edges in the inner dual. Also, by allowing trimers
of the form $\left\lbrace i-1,i,i+1\right\rbrace$ we abandon the context
of matchings and instead work with packings in the inner dual.
The allowed tiles are illustrated in Figure \ref{Tiles tribonacci}.

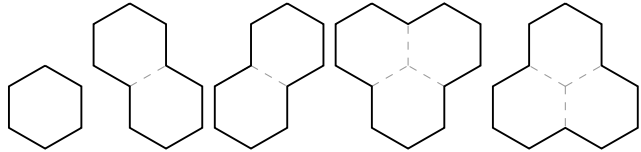
\begin{figure}[h!]
\centering \begin{tikzpicture}[scale=0.55]
\draw [line width=0.25mm](A)--(B)--(C)--(D)--(E)--(F)--(A); 
\end{tikzpicture}  \begin{tikzpicture}[scale=0.55]

\draw [line width=0.25mm](A)--(B)--(C)--(D)--(I)--(J)--(K)--(L)--(E)--(F)--(A);
\draw[dashed,opacity=0.4] (D)--(E); 
\end{tikzpicture}  \begin{tikzpicture}[scale=0.55]

\draw[dashed] (D) -- (E); 
\draw [line width=0.25mm] (A1)--(B1)--(C1)--(H)--(I)--(J)--(K)--(L)--(E1)--(F1)--(A1); 
\draw[dashed,opacity=0.4] (G)--(L); 
\end{tikzpicture}  \begin{tikzpicture}[scale=0.55]

\draw [line width=0.25mm] (A) -- (B) -- (C) -- (D) -- (I) -- (J) --(K)--(L)--(E1)--(F1)--(A1)--(B1)--(A); 
\draw[dashed,opacity=0.4] (D)--(E) (L)--(E) (E)--(F) ; 
\end{tikzpicture}  \begin{tikzpicture}[scale=0.55]

\draw [line width=0.25mm] (A1)--(B1)--(C1)--(H)--(I)--(J)--(K)--(J1)--(K1)--(L1)--(G1)--(F1)--(A1); 
\draw[dashed,opacity=0.4] (L)--(G) (L)--(K) (L)--(E1); 
\end{tikzpicture}  
\caption{The allowed types of tiles.}
\label{Tiles tribonacci}
\end{figure} 

Let $g_n$ denotes the number  of ways to tile a hexagonal strip of length
$n$ by using only the allowed tiles. It is convenient to define $g_0=1$, and
it is immediately clear that $g_1=1$, $g_2=2$.

\begin{tm} Let $g_n$ denote the number of all ways to tile a honeycomb
strip of length by using only the allowed tiles. Then
$$g_n=T_{n+2},$$
where $T_n$ denotes $n$-th tribonacci number.
\end{tm} 
\begin{proof}
We start with an arbitrary tiling of a strip. There are three disjoint
cases involving the $n$-th hexagon. If the hexagon is tiled by a monomer,
then the rest of the strip can be tiled in $g_{n-1}$ ways. If it is covered by
a dimer, there are $g_{n-2}$ such tilings, and finally, if the rightmost
hexagon is covered by a trimer, the are $g_{n-3}$ such tilings. By summing
the respective numbers we obtain a recurrence that is the same as the
defining recurrence for the tribonacci numbers, and the initial values
determine the value of the shift.
\end{proof}
 
In the next part, we refine our results by counting the number of tilings with
a fixed number of trimers, dimers or monomers. We denote these numbers
by $t_{n,k}$, $u_{n,k}$ and $v_{n,k}$, respectively, where $n$, as usual,
denotes the length of a strip, and $k$ the number of tiles of a certain
kind. We can also fix the number of all types of tiles. Let $g_{n}^{k,l}$
denotes the number of all ways to tile a strip of a length $n$ using exactly
$k$ trimers, $l$ dimers and $n-3k-2l$ monomers. We list some first values
in the Table \ref{Initial value t_n,k}. From the definition it is clear
that $t_{n,k}=0$ for $k>\floor*{\frac{n}{3}}$, $u_{n,k}=0$ for
$k>\floor*{\frac{n}{2}}$ and $v_{n,k}=0$ for $k>n$. It is also convenient
to define $t_{0,0}=u_{0,0}=v_{0,0}=1$. For these sequences we can obtain
recursive relations in the obvious way, by considering the state of the last
hexagon to see whether it is covered by a trimer, by a dimer, or by a 
monomer. The recursive relations are: \begin{equation}
t_{n,k}=t_{n-1,k}+t_{n-2,k}+t_{n-3,k-1},\end{equation} \begin{equation}
u_{n,k}=u_{n-1,k}+u_{n-2,k-1}+u_{n-3,k}\end{equation} and \begin{equation}
v_{n,k}=v_{n-1,k-1}+v_{n-2,k}+v_{n-3,k}.\end{equation}

We can now list some first values of the corresponding triangles: 

\begin{table}[h!]\centering
$\begin{array}{ccc}
\begin{array}{c|ccc}
n/k & 0 & 1 \\
\hline
0 & 1\\
1 & 1\\
2 & 2\\
3 & 3& 1\\
4 & 5& 2 \\
5 & 8& 5\\
6 & 13 & 10 &1\\
7 & 21 & 20 & 3\\
8 & 34 & 38 & 9 
\end{array} & \begin{array}{c|ccccc}
n/k & 0 & 1 \\
\hline
0 & 1 \\
1 & 1 \\
2 & 1& 1\\
3 & 2& 2\\
4 & 3 & 3 & 1\\
5 & 4 & 6 & 3\\
6 & 6 &11 & 6 & 1\\
7 & 9 &18 & 13 & 4\\
8 & 13&30 & 27 & 10 & 1\\
\end{array} & \begin{array}{c|ccccccccc}
n/k & 0 & 1 & 2 & 3 &4 &5& 6 &7& 8\\
\hline
0 & 1\\
1 & 0&1 \\
2 & 1&0&1\\
3 & 1&2&0&1\\
4 & 1&2&3&0&1\\
5 & 2&3&3&4&0&1\\
6 & 2&6&6&4&5&0&1\\
7 & 3&7&12&10&5&6&0&1\\
8 & 4&12&16&20&15&6&7&0&1\\
\end{array}\\
t_{n,k} & u_{n,k} & v_{n,k} \\
\end{array}$
\caption{Initial values for $t_{n,k}$, $u_{n,k}$ and $v_{n,k}$.}
\label{Initial value t_n,k}
\end{table} 

The first and the second triangle of Table \ref{Initial value t_n,k}
are not in the OEIS, while the third one appears as A104578 \cite{OEIS},
the Padovan convolution triangle.
The same arguments as the ones used on $c_{n,k}$ shows that the rows
of those triangles do not have internal zeros, with the obvious exception
of the zeros appearing in the first descending subdiagonal of $v_{n,k}$.

Before we go any further, we introduce two closely related sequences
defined by Fibonacci-like recurrences of length three, namely the
Narayana's cows sequence (A000930) and the Padovan sequence (A000931).
We denote the $n$-th element of these sequences by $N_n$ and $P_n$,
respectively. The initial values are 
$N_{0}=N_{1}=N_{2}=1$ and $P_{0}=1$, $P_{1}=P_{2}=0$, and for $n\geq 3$ we
have recursive relation $N_n=N_{n-1}+N_{n-3}$ and $P_n=P_{n-2}+P_{n-3}$. 
We refer the reader to \cite{OEIS} for more details about those sequences.
In particular, we draw the reader's attention to the fact that there are several
other sequences referred to as the Narayana numbers, for example A001263,
a very important triangle of numbers refining the Catalan numbers and appearing
in many different contexts. In the rest of this paper, when we refer to
Narayana's numbers, we always mean A000930. 

We now take a closer look at sequences $t_{n,0}$, $u_{n,0}$ and
$v_{n,0}$, i.e., at the number of tilings where one type of tile is omitted.
The sequence $t_{n,0}=F_{n+1}$, since such tilings contain only slanted dimers
and monomers; since such tilings correspond to matchings in the path on $n$
vertices, they are counted by Fibonacci numbers.

The sequence $u_{n,0}$ counts the number of all ways to tile a length-$n$ strip by using only monomers and trimers, hence its elements satisfy the defining recurrence for the Narayana's cow sequence. Similarly, since the elements of the sequence $v_{n,0}$ are the numbers of all different tiling where monomers are omitted, they satisfy Padovan's recursion. We have $u_{n,0}=N_n$ and $v_{n,0}=P_{n+3}$. In the next three theorems we present connection between the number of tilings and above listed sequences. It turns out that the elements of the three triangles of Table \ref{Initial value t_n,k} can be expressed by convolution-like formulas involving the Fibonacci, the Narayana's and the Padovan numbers. Such formulas could have been anticipated from the second column of triangle $t_{n,k}$ which seems to be the (shifted) self-convolution of Fibonacci numbers and also from the name of the entry A104578 in OEIS.

\begin{tm}\label{thm t_n,k} For $n\geq 0$, the number of ways to tile
a strip with $n$ hexagons using exactly $k$ trimers is
\begin{equation}
t_{n,k}=\sum\limits_{\substack{i_0,\dots, i_{k}\geq 0\\ i_0+\cdots+i_{k}=n-2k+1}}F_{i_0}\cdots F_{i_{k}} .
\end{equation}  
\end{tm}
\begin{proof}
If there are no trimers in the tiling, one can only use dimers or monomers
to tile a strip and the number ways to do that is $t_{n,0}=F_{n+1}$.
If we use exactly $k$ trimers, those trimers divide our strip into
$k+1$ smaller strips. In this sense we allow the strip to be of a length
$0$ if two trimers are adjacent; the sub-strips of length 0 can also appear
at the beginning or at the end of a strip.
We have a strip with $n$ hexagons which is tiled with $k$ trimers, so there
are $n-3k$ hexagons left to tile. Since the position of each trimer is
arbitrary, the lengths of strips between and around them can vary from $0$
to $n-3k$, but the sum of the lengths must be constant, that is
$i_0+i_1+\dots+i_{k}=n-3k$. Each of those smaller strips can be tiled
only by dimers or monomers, hence in $t_{i_{j},0}$ ways, where
$0\leq j\leq k$. Summing this over all positions of the trimers we have: 
\begin{align*}
t_{n,k}=&\sum\limits_{\substack{i_0,\dots, i_{k}\geq 0\\ i_0+\cdots+i_{k}=n-3k}}t_{i_0,0}\cdots t_{i_{k},0}\\ 
=&\sum\limits_{\substack{i_0,\dots, i_{k}\geq 0\\ i_0+\cdots+i_{k}=n-3k}}F_{i_0+1}\cdots F_{i_{k}+1}\\ 
=&\sum\limits_{\substack{i_0,\dots, i_{k}\geq 0\\ i_0+\cdots+i_{k}=n-2k+1}}F_{i_0}\cdots F_{i_{k}} .
\end{align*}
\end{proof}
Note that Theorem \ref{thm t_n,k} allow us to express the tribonacci numbers
as a double sum \begin{equation}
T_{n+2}=\sum_{k=0}^nt_{n,k}=\sum_{k=0}^n\sum\limits_{\substack{i_0,\dots, i_{k}\geq 0\\ i_0+\cdots+i_{k}=n-2k+1}}F_{i_0}\cdots F_{i_{k}} .
\end{equation} 

\begin{tm}\label{thm u_n,k} For $n\geq 0$, the number of ways to tile
a strip with $n$ hexagons using exactly $k$ dimers is
\begin{equation}
u_{n,k}=\sum\limits_{\substack{i_0,\dots, i_{k}\geq 0\\ i_0+\cdots+i_{k}=n-2k}}N_{i_0}\cdots N_{i_{k}} .
\end{equation}  
\end{tm}
\begin{proof} We already know that the number of tilings with no dimers
is $u_{n,0}=N_n$. Now we look at the tilings of the strip with $n$ hexagons
that have exactly $k$ dimers. That leaves us with $n-2k$ hexagons to be
tiled by monomers and trimers. As in the proof of Theorem \ref{thm t_n,k},
we note that $k$ dimers divide the strip into $k+1$ smaller strips,
each of the length $0\leq i_j\leq n-2k$. Each smaller strip can be tiled in
$N_{i_j}$ ways, and after summing over all possible positions of
$k$ dimers we have \begin{equation*}
u_{n,k}=\sum\limits_{\substack{i_0,\dots, i_{k}\geq 0\\ i_0+\cdots+i_{k}=n-2k}}N_{i_0}\cdots N_{i_{k}}. 
\end{equation*} 
\end{proof}
The next result gives a new combinatorial interpretation of sequence
A104578 of \cite{OEIS}.
\begin{tm}\label{thm v_n,k} For $n\geq 0$, the number of ways to tile a
strip with $n$ hexagons using exactly $k$ monomers is \begin{equation}
v_{n,k}=\sum\limits_{\substack{i_0,\dots, i_{k}>0\\ i_0+\cdots+i_{k}=n+2k+3}}P_{i_0}\cdots P_{i_{k}} .
\end{equation}  
\end{tm}
\begin{proof} The proof will be analogous to the two previous proofs.
The number of tilings with no monomers is $v_{n,0}=P_{n+3}$. A monomer
does not divide our strip, but if it first appears in position $i$, we
will consider strips left and right from it. We count the number
of tilings of the strip $H_n$ that have exactly $k$ monomers.
That leaves us $n-k$ untiled hexagons. Omitting $k$ hexagons leaves us
with with $k+1$ smaller strips, each of the length $0\leq i_j\leq n-k$.
Each smaller strip can be tiled in $P_{i_j+3}$ ways, and after summing
over all possible positions of $k$ monomers we have \begin{align*}
v_{n,k}=&\sum\limits_{\substack{i_0,\dots, i_{k}\geq 0\\ i_0+\cdots+i_{k}=n-k}}P_{i_0+3}\cdots P_{i_{k}+3}\\ 
=&\sum\limits_{\substack{i_0,\dots, i_{k}>0\\ i_0+\cdots+i_{k}=n+2k+3}}P_{i_0}\cdots P_{i_{k}}. \end{align*} \end{proof}

Now we turn our attention to the number of tilings of a strip of length $n$
with numbers of all types of tiles fixed. Recall that number of tiling
consisting of $k$ trimers, $l$ dimers and $n-3k-2l$ monomers is denoted by
$g ^{k,l}_n$.  
In the next theorem we give a closed form formula for $g^{k,l}_n$.

\begin{tm}\label{thm t_n,k,l} For $n\geq 0$, the number of ways to tile
a strip with $n$ hexagons using exactly $k$ trimers, $l$ dimers and
$n-2k-l$ monomers is \begin{equation}
g^{k,l}_{n}=\binom{n-3k-l}{l}\binom{n-2k-l}{k}.
\end{equation} 
\end{tm}
\begin{proof}
Consider a set consisting of all arbitrary tilings of a length-$n$ strip that
have exactly $k$ trimers, $l$ dimers and $n-3k-2l$ monomers. To prove this
theorem, we establish a 1-1 correspondence between that set and the set of
all permutations of $n-2k-l$ elements where we have $k$ elements $t$, $l$
elements $d$ and $n-3k-2l$ elements $m$. From an arbitrary permutation we
obtain the corresponding tiling as follows: we replace each element $t$ with 
a trimer, each element $d$ with a dimer, and each element $m$ with a monomer.
In this manner we obtained a tiling of a strip of length $n$ with prescribed
number of tiles of each type. For example, the permutation $tmdmt$
corresponds to the tiling shown in the Figure \ref{tmdmt}.
\begin{figure}[h!]\centering
\begin{tikzpicture}[scale=0.55]

\draw [line width=0.25mm] (A1)--(B1)--(C1)--(H)--(I)--(J)--(K)--(J1)--(K1)--(L1)--(G1)--(F1)--(A1); 
\draw [line width=0.25mm] (A2)--(B2)--(C2)--(D2)--(E2)--(F2)--(A2); 
\draw [line width=0.25mm] (A3)--(B3)--(C3)--(H2)--(I2)--(J2)--(K2)--(L2)--(E3)--(F3)--(A3); 
\draw [line width=0.25mm] (G3)--(H3)--(I3)--(J3)--(K3)--(L3)--(G3); 
\draw [line width=0.25mm] (A4)--(B4)--(C4)--(H4)--(I4)--(J4)--(K4)--(L4)--(E5)--(F5)--(A5)--(B5)--(A4); 
\draw[dashed,opacity=0.4] (C1)--(D1)--(E1) (D1)--(K) (C3)--(D3) (D4)--(E4)--(D5) (E4)--(F4); 
\end{tikzpicture}  
\caption{Tiling corresponding to the permutation $tmdmt$.} \label{tmdmt}
\end{figure}
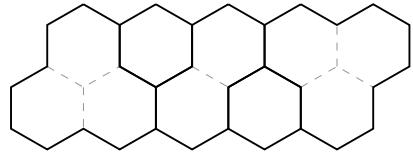 
In an obvious way we can also obtain a permutation from a given tiling.
Since the total number of permutations of this set is
$\frac{(n-2k-l)!}{k!l!(n-3k-2l)!}$, we arrive to:
\begin{align*}
g^{k,l}_n&=\frac{(n-2k-l)!}{k!l!(n-3k-2l)!}\cdot\frac{(n-3k-l)!}{(n-3k-l)!}\\
&=\frac{(n-3k-l)!}{l!(n-3k-2l)!}\cdot\frac{(n-2k-l)!}{k!(n-3k-l)!}\\
&=\binom{n-3k-l}{l}\binom{n-2k-l}{k}.
\end{align*}
 \end{proof}
 
From Theorem \ref{thm t_n,k,l} we arrive to yet another identity for
tribonacci numbers: \begin{equation}
T_{n+2}=\sum\limits_{k=0}^{\floor*{\frac{n}{3}}}\sum\limits_{l=0}^{\floor*{\frac{n-3k}{2}}}\binom{n-3k-l}{l}\binom{n-2k-l}{k} \end{equation} 
Specially, if we set $k=0$ in the first equation we have \begin{align*}
t_{n,0}&=\sum\limits_{l=0}^{\floor*{\frac{n}{2}}}\binom{n-l}{l}\binom{n-l}{0}\\
&=\sum\limits_{l=0}^{\floor*{\frac{n}{2}}}\binom{n-l}{l}\\
&=F_{n+1}.
\end{align*}
Since Theorem \ref{thm t_n,k,l} gives us the number of all tilings using 
the prescribed number of tiles of each type, we can express values $t_{n,k}$
and $u_{n,k}$ in a new way by summing over $l$ and $k$, respectively.

\begin{cor} For $n\geq 0$, \begin{equation}
t_{n,k}=\sum\limits_{l=0}^{\floor*{\frac{n-3k}{2}}}\binom{n-3k-l}{l}\binom{n-2k-l}{k},
\end{equation} and \begin{equation}
u_{n,l}=\sum\limits_{k=0}^{\floor*{\frac{n-2l}{3}}}\binom{n-3k-l}{l}\binom{n-2k-l}{k}.
\end{equation}   
\end{cor}

\section{Some identities involving tribonacci numbers}

In this section we prove, in a combinatorial way, several identities involving 
the tribonacci, Narayana's, Padovan and Fibonacci numbers. We begin with a well-known identity for tribonacci numbers and we give it a
new combinatorial interpretation: 
\begin{tm} For $n\geq 4$, $$T_n+T_{n-4}=2T_{n-1}.$$\end{tm}
\begin{proof} Let $\mathcal{G}_n$ denotes the set of all tilings of a length-$n$ strip,  $\mathcal{M}_n$, $\mathcal{D}_n$ and $\mathcal{T}_n$ the tilings ending with a monomer, dimer or trimer, respectively. As before, the cardinal number of the set $\mathcal{G}_n$ is $g_n$. It is clear that $\mathcal{T}_n=\mathcal{M}_n\dot\cup\mathcal{D}_n\dot\cup\mathcal{T}_n$. To prove the theorem we have to establish 1-1 correspondence between sets $\mathcal{G}_{n-2}\cup\mathcal{G}_{n-6}$ and $\mathcal{G}_{n-3}\times\left\lbrace0,1\right\rbrace$.

To each tiling from the set $\mathcal{G}_{n-3}$ we add a monomer at the end to obtain an element of $\mathcal{M}_{n-2}$. Thus, we obtained bijection between the sets $\mathcal{G}_{n-3}$ and $\mathcal{M}_{n-2}$. In this way, we have used all the tiling of the set $\mathcal{G}_{n-3}$ once. Now we take the tilings from the set $\mathcal{G}_{n-3}$ again, and if it ends with a trimer, i.e. if it is an element of $\mathcal{T}_{n-3}$, we remove it
to obtain a tiling of length $n-6$, i.e., element of a set $\mathcal{G}_{n-6}$. If it ends with a dimer (an element of $\mathcal{D}_{n-3}$), we remove it
and replace it with a trimer to obtain element from $\mathcal{T}_{n-2}$.  Finally, if the tiling is an element of $\mathcal{M}_{n-3}$, we replace the last monomer with a
dimer, to obtain an element of $\mathcal{D}_{n-2}$. In this way we have used every tiling of a length $n-3$ twice and obtained all tilings of a length $n-2$ and $n-6$ exactly once. Diagram that visualize 1-1 correspondence between the two sets is shown in the Figure \ref{1-1 cor}.  

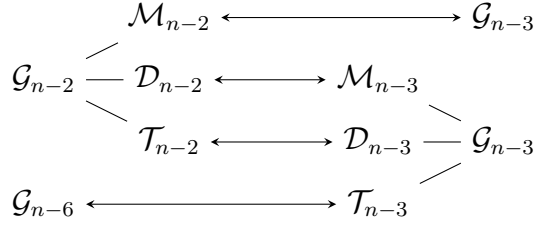
\begin{figure}[h!]\centering
\begin{tikzpicture}[scale=0.55]

\node[] (p1) at (0,0) {$\mathcal{G}_{n-2}$};
\node[] (p2) at (3,1.5) {$\mathcal{M}_{n-2}$};
\node[] (p3) at (3,0) {$\mathcal{D}_{n-2}$};
\node[] (p4) at (3,-1.5) {$\mathcal{T}_{n-2}$};
\node[] (p5) at (0,-3) {$\mathcal{G}_{n-6}$};
\node[] (p6) at (11,1.5) {$\mathcal{G}_{n-3}$};
\node[] (p7) at (11,-1.5) {$\mathcal{G}_{n-3}$};
\node[] (p8) at (8,0) {$\mathcal{M}_{n-3}$};
\node[] (p9) at (8,-1.5) {$\mathcal{D}_{n-3}$};
\node[] (p10) at (8,-3) {$\mathcal{T}_{n-3}$};
\draw [stealth-stealth](p2)--(p6);
\draw [stealth-stealth](p3)--(p8);
\draw [stealth-stealth] (p4)--(p9);
\draw [stealth-stealth] (p5)--(p10);
\draw (p1)--(p2) (p1)--(p3) (p1)--(p4) (p7)--(p8) (p7)--(p9) (p7)--(p10);

\end{tikzpicture}  
\caption{1-1 correspondence between sets $\mathcal{G}_{n-2}\cup\mathcal{G}_{n-6}$ and $\mathcal{G}_{n-3}\times\left\lbrace0,1\right\rbrace$.} \label{1-1 cor}
\end{figure} 

It follows that $g_{n-2}+g_{n-6}=2g_{n-3}$,
and since $g_{n}=T_{n+2}$, the theorem follows.\end{proof}

For the next few identities is is useful to recall the definition of
breakability. We say that a tiling of a honeycomb strip is breakable
at the position $k$ if given tiling can be divided into two tiled
strips, the first one containing the leftmost $k$ hexagons and the second 
one containing the rest.   

Our next identity differentiates tilings based on the breakability.
\begin{tm} For any integers $m,n\geq 1$ we have the identity
$$T_{m+n}=T_{m}T_{n}+T_{m+1}T_{n+1}+T_{m-1}T_{n}+T_{m}T_{n-1}.$$
\end{tm}
\begin{proof}
We consider an arbitrary tiling of a strip of length
$m+n-2$. If the tiling is breakable at position $m-1$, we divide
it into two strips of a length $m-1$ and $n-1$. Hence, the total number
of tiling in this case is $g_{m-1}g_{n-1}$. If the tiling is not breakable
at position $m-1$, that means that either a dimer or a trimer is blocking it.
If the dimer is preventing the tiling from breaking, there are strips of
lengths $m-2$ and $n-2$ on each side, so the total number of tilings in this
case is $g_{m-2}g_{n-2}$. If the trimer is blocking it, it can reduce
the length of the left or of the right strip by two. So the total number
of tilings in this case is $g_{m-3}g_{n-2}+g_{m-2}g_{n-3}$.
By summing the contributions of all these cases we obtain
$g_{m+n-2}=g_{m-1}g_{n-1}+g_{m-2}g_{n-2}+g_{m-3}g_{n-2}+g_{m-2}g_{n-3}$,
and by using the equality $g_n=T_{n+2}$ we have
$T_{m+n}=T_{m}T_{n}+T_{m+1}T_{n+1}+T_{m-1}T_{n}+T_{m}T_{n-1}$. 
\end{proof}  

The next identity was proved by Frontczak \cite{Frontczak} by using generating functions. Here we provide a combinatorial interpretation. 

\begin{tm}
For any integer $n\geq 0$ we have the identity
$$T_{n+2}=\sum\limits_{k=0}^{n+1}F_kT_{n-k}.$$
\end{tm}
\begin{proof}
We prove this theorem by counting all ways to tile a strip by using at
least one trimer. The total number of ways to tile a length-$n$ strip without
trimers is $t_{n,0}=F_{n+1}$, hence the number of tilings having at least
one trimer is $T_{n+2}-F_{n+1}$. On the other hand, we can count the same
tilings by observing where the first trimer appears. If the leftmost
trimer occupies
hexagons $\left\lbrace i,i+1,i+2\right\rbrace$, we say that the position
of trimer is $i$. So, all possible positions range from $1$ to $n-2$.
If a trimer first appears at position $k$, the leftmost $k-1$ hexagons
are tiled only by monomers and dimers, and the number of all ways
to do that is $F_k$. The rest of the strip, of length $n-k-2$, can be tiled
in $T_{n-k}$ ways. By summing over all possible positions of the leftmost
trimer, we have $T_{n+2}-F_{n+1}=\sum\limits_{k=1}^{n-2}F_{n}T_{n-k}$. By using
$T_{n-3}=T_{n}-T_{n-1}-T_{n-2}$, one can extend the tribonacci numbers
to negative integers and obtain $T_{-1}=1$. Since $T_{1}=T_{0}=0$,
the sum above can be extended to obtain
$T_{n+2}=\sum\limits_{k=1}^{n+1}F_{k}T_{n-k}$,
which concludes our proof.  
\end{proof}

\begin{tm} For any integer $n\geq 0$ we have the identity
$$T_{n+2}=\sum\limits_{k=0}^{n}N_{k}T_{n-k}+N_{n}.$$ \end{tm}
\begin{proof}
We prove this theorem by counting all ways to tile a strip by using
at least one dimer. The proof is analogous to the previous one. The total
number of ways to tile a strip of length $n$ without dimers is
$u_{n,0}=N_n$, hence the number of tilings having at least one dimer
is $T_{n+2}-N_n$. Similarly as before, we can count the same thing by
observing where the leftmost dimer appears. If the leftmost dimer occupies
hexagons $\left\lbrace i,i+1\right\rbrace$, we say that its position
is $i$. So, all possible positions range from $1$ to $n-1$. If a dimer
first appears at position $k$, the leftmost $k-1$ hexagons
can be tiled in $N_{k-1}$ ways. The rest of the strip is of length $n-k-1$
and it can be tiled in $T_{n-k+1}$ ways. By summing over all possible positions
of the leftmost dimer we have
$T_{n+2}-N_{n}=\sum\limits_{k=1}^{n-1}N_{k-1}T_{n-k+1}$.
Some rearranging of indexes and fact that $T_0=T_1=0$ bring us to
$T_{n+2}-N_{n}=\sum\limits_{k=0}^{n}N_{k}T_{n-k}$ and our proof is over.  
\end{proof}
\begin{tm}
For any integer $n\geq 0$ we have the identity
$$T_{n+2}=\sum\limits_{k=1}^{n}P_{k+2}T_{n-k+2}+P_{n+3}.$$
\end{tm}
\begin{proof}
Analogously as in two previous theorems, we prove this
theorem by counting all ways to tile a strip by using at least one monomer.
The number of ways to tile a strip of length $n$ with at least one monomer is
$g_n-v_{n,0}=T_{n+2}-P_{n+3}$. Now we can count the same thing by observing
the position of the leftmost monomer. All possible positions for first monomer
range from $1$ to $n$. If it first appears at position $k$, the first part
of the strip, i.e., the leftmost $k-1$ hexagons, can be tiled in
$P_{k+2}$ ways. The rest of the
strip is of length $n-k$ and can be tiled in $T_{n-k+2}$ ways.
By summing over all possible positions of the leftmost monomer we have
$T_{n+2}-P_{n+3}=\sum\limits_{k=1}^{n}P_{k+2}T_{n-k+2}$, which concludes
our proof. 
\end{proof}

\section{Concluding remarks}

In this paper we have considered various ways of tiling a narrow honeycomb
strip of a given length with different types of tiles. We have refined some
previously known results for the total number of tilings of a given type
by deriving formulas for the number of such tilings with prescribed
number of tiles of a given type. We have also considered tilings with
colored tiles and obtained the corresponding formulas. Along the way,
we have provided combinatorial interpretations for some known identities
and established a number of new ones. Also, we have provided closed-form
expressions for several triangles of numbers appearing in the OEIS.

In order to keep this contribution at a reasonable length, we have omitted
many interesting problems related to the considered ones. In particular,
we have not considered any jamming-related scenarios, i.e., the tilings 
which are suboptimal with respect to the number of large(r) tiles. The
existence of connections of our tilings with such problems is indicated
by the appearance of Padovan numbers in both contexts \cite{dosliczub}.
Further, we have not examined statistical properties such as the expected
number of tiles in a random tiling of a strip of a given length in a way
done in ref. \cite{doslicasr}. We have not looked
at the asymptotic behavior of the counting sequences. Each of the mentioned
omissions could be an interesting topic for further research.

Another interesting direction would be to look in more detail at triangles
$c_{n,k}$, $t_{n,k}$, $u_{n,k}$, and $v_{n,k}$. We have shown that their
rows (with one trivial exception) do not have internal zeros. By inspection
of the first few rows of $c_{n,k}$, $t_{n,k}$, and $u_{n,k}$ one can
observe that the rows seem to be also unimodal and even log-concave. It would
be interesting to investigate whether those properties hold for the whole
triangles. Both properties are violated in rows of $v_{n,k}$, but the
violations seem to be restricted to the right end. What happens if the rightmost
three elements are omitted? Also, the position of maximum presents an
interesting challenge. Finally, more interesting identities
could be derived by looking at ascending and descending diagonals of
different slopes in those triangles. We hope that at least some of the
listed problems would be addressed in our future work.

\section*{Acknowledgment}
T. Do\v{s}li\'c gratefully acknowledges partial support
of Slovenian ARRS via grants J1-3002 and P1-0383.

\end{document}